\documentclass[11pt]{amsart}
\title[Shadow-complexity and trisection genus]
{Shadow-complexity and trisection genus}

\author{Hironobu Naoe}
\address{
Department of Mathematics, Tokyo Institute of Technology, 2-12-1 Ookayama, Meguro-ku, Tokyo, 152-8551, Japan}
\email{naoe@math.titech.ac.jp}

\author{Masaki Ogawa}
\address{
Mathematical Science Center for Co-creative Society, Tohoku University, Aoba-6-3 Aramaki, Aoba-ku, Sendai, Miyagi, 980-0845, Japan}
\email{masaki.ogawa.b7@tohoku.ac.jp}

\subjclass[2020]{Primary 57K41, 57Q15; Secondary 57R65.}
\usepackage{amsfonts,amsmath,amssymb,amscd}
\usepackage{amsthm}
\usepackage{latexsym}
\usepackage{bm}
\usepackage[dvipdfmx]{graphicx}
\usepackage{psfrag}
\usepackage{color}
\usepackage{url}
\usepackage{ulem}
\usepackage{pinlabel}

\theoremstyle{plain}
\newtheorem{theorem}{Theorem}[section]
\newtheorem{lemma}[theorem]{Lemma}

\newtheorem{proposition}[theorem]{Proposition}

\theoremstyle{definition}
\newtheorem{definition}[theorem]{Definition}

\newtheorem{remark}[theorem]{Remark}

\newtheoremstyle{mycitation}%   
    {}%                      
    {}%                    
    {\it}%          
    {}%                  
    {\bf}%             
    {}%                    
    {5pt}%                
    {\thmname{#1} \thmnumber{#2}\thmnote{#3}.}%
\theoremstyle{mycitation}
\newtheorem*{citingtheorem}{Theorem}

\newcommand{\Z}{\mathbb{Z}}
\newcommand{\R}{\mathbb{R}}
\newcommand{\C}{\mathbb{C}}
\newcommand{\CP}{\mathbb{CP}^2}
\newcommand{\mCP}{\overline{\mathbb{CP}}^2}

\newcommand{\RP}{\mathbb{RP}}

\newcommand{\Int}{\mathrm{Int}}

\newcommand{\Nbd}{\mathrm{Nbd}}

\newcommand{\gl}{\mathfrak{gl}}

\newcommand{\shco}{\mathrm{sc}}
\newcommand{\spshco}{\mathrm{sc}^{\mathrm{sp}}}

\definecolor{darkred}{rgb}{.5,.0,.0}
\definecolor{darkblue}{rgb}{.0,.0,.70}

\definecolor{bl}{gray}{0.7}

\definecolor{myred}{rgb}{.8,.0,.0}
\definecolor{mygreen}{rgb}{.0,.6,.0}
\definecolor{mygray}{gray}{0.7}

\newlength{\myheight}
\newlength{\myheighta}
\makeatletter

\long\def\@makecaption#1#2{
  \small
  \vskip\abovecaptionskip
  \sbox\@tempboxa{#1. #2}
  \ifdim \wd\@tempboxa >\hsize
    #1. #2\par
  \else
    \global \@minipagefalse
    \hb@xt@\hsize{\hfil\box\@tempboxa\hfil}
  \fi
  \vskip\belowcaptionskip}

\makeatother
\begin{document}
\begin{abstract}
The shadow-complexity is an invariant of closed $4$-manifolds defined by using $2$-dimensional polyhedra called Turaev's shadows, 
which, roughly speaking, measures how complicated a $2$-skeleton of the $4$-manifold is. 
In this paper, we define a new version $\mathrm{sc}_{r}$ of shadow-complexity depending on an extra parameter $r\geq0$, 
and we investigate the relationship between this complexity and the trisection genus $g$. 
More explicitly, we prove an inequality $g(W) \leq 2+2\mathrm{sc}_{r}(W)$ for any closed $4$-manifold $W$ and any $r\geq1/2$. 
Moreover, we determine the exact values of $\mathrm{sc}_{1/2}$ for infinitely many $4$-manifolds, 
and also we classify all the closed $4$-manifolds with $\mathrm{sc}_{1/2}\leq1/2$. 
\end{abstract}
\maketitle

%==================================================================================
\section{Introduction}
% Combinatorial methods of representing $4$-manifolds have been given, for example, 
% Kirby diagrams, trisection diagrams, and Turaev's shadows. 
% Using such descriptions, one can define invariants as the minimums of certain complexities of the descriptions. 
% The main purpose of this paper is to compare them with each other. 
%
A \textit{shadow} is a locally-flat simple polyhedron embedded in a connected closed oriented smooth $4$-manifold as a $2$-skeleton
(see Definition~\ref{def:shadow}), which was introduced by Turaev for the purpose of studying quantum invariants \cite{Tur94}. 
Afterwards, Costantino provided some applications of shadows to the topology of $3$- and $4$-manifolds. 
For example, we refer the reader to \cite{Cos06,Cos08} for the studies of Stein structures, 
Spin${}^c$ structures and almost complex structures of connected oriented smooth $4$-manifolds with boundary. 
In \cite{Cos06b}, he defined invariants of $3$- and $4$-manifolds called the \textit{shadow-complexity} $\shco$ and 
the \textit{special shadow-complexity} $\spshco$ as the minimum numbers of certain vertices called true vertices of shadows of a fixed manifold. 
% These can be considered as analogues of the Matveev complexity of $3$-manifolds. 
The shadow-complexity of $3$-manifolds is closely related with the Gromov norm and stable maps of $3$-manifolds \cite{CT08,IK17},
which provided a geometric perspective on the shadow-complexity of $3$-manifolds. 
% Especially, the set of $3$-manifolds with shadow-complexity zero coincides with that of graph manifolds. 
% \Erase{Particularly, the shadow-complexity for $3$-manifolds is studied by comparing it with another invariant, namely Gromov norm, of $3$-manifolds, 
% and its geometric perspective has been established.} 
In contrast to such studies, 
the shadow-complexity for $4$-manifolds has been studied about the classification problem \cite{Cos06b,KMN18,Mar11,Nao17,Nao23}. 
This paper aims to investigate a behavior of the shadow-complexity of $4$-manifolds, 
and we provide a comparison between it and the \textit{trisection genus} in particular. 

A \textit{trisection} is a decomposition of connected closed oriented smooth $4$-manifold 
into three $4$-dimensional $1$-handlebodies (see Definition~\ref{def:trisection} for the precise definition). 
The intersection of the three portions forms a surface, which is called the central surface of the trisection. 
The \textit{trisection genus} $g$ of a $4$-manifold is defined as the minimum genus of central surfaces of trisections of the $4$-manifold, 
and $g$ is of course an invariant of $4$-manifolds. 
Only the $4$-sphere is the closed $4$-manifold with $g=0$, 
and only $\pm\CP$ and $S^1\times S^3$ are those with $g=1$. 
The $4$-manifolds with $g=2$ were classified by Meier and Zupan \cite{MZ17}. 
The cases of $g\geq3$ are still open, and Meier conjectured in \cite{Mei18} that 
an irreducible $4$-manifold with $g=3$ is either $\mathcal{S}_p$ or $\mathcal{S}'_p$ for some integer $p\geq2$, 
where $\mathcal{S}_p$ and $\mathcal{S}'_p$ are $4$-manifolds obtained from $S^1\times S^3$ 
by surgering along a simple closed curve representing $p\in\Z\cong \pi_1(S^1\times S^3)$. 
We also refer the reader to \cite{Wil20} for the decision of the trisection genera of trivial surface bundles over surfaces. 

In this paper, we define a new kind of shadow-complexity called the $r$-\textit{weighted shadow-complexity} $\shco_r$ for each fixed $r\in\R_{\geq0}$, 
which is an invariant of $4$-manifolds. 
It takes a value in $\{m+rn\mid m,n\in\Z_{\geq0}\}$. 
The weighted shadow-complexity is defined by minimizing the sum of the number of true vertices and a ``complexity'' of regions of shadows, 
although 
% \Erase{the shadow-complexity is determined only by the number of true vertices. }
we consider only the number of true vertices with regard to the shadow-complexity. 

We establish a method to construct a trisection from a given shadow of a closed $4$-manifold via a Kirby diagram. 
This method includes how to describe a trisection diagram, 
and it allows us to estimate the trisection genus of the $4$-manifold from the combinatorial information of the shadow. 
The following is the main theorem in this paper. 
\begin{citingtheorem}
[\ref{thm:complexity_genus}]
For any closed $4$-manifold $W$ and any real number $r\geq1/2$, 
$g(W)\leq 2 + 2\shco_r(W)$. 
\end{citingtheorem}
The equality $g(W) = 2 + 2\shco_{1/2}(W)$ is attained, for instance, by 
% $W=S^2\times S^2$, $\CP\#\CP$, $\CP\#\mCP$, $\mCP\#\mCP$, $\mathcal{S}_2$, ${\mathcal{S}}'_2$, $\mathcal{S}_3$. 
$W=k_1(S^2\times S^2)\# k_2\CP\# k_3\mCP$ for any $k_1,k_2,k_3\in\Z_{\geq0}$. 
In this sense, we can say that the inequality is the best possible (cf. Remark~\ref{rmk:best_result}). 

We compare the $3$ series of the shadow-complexities 
$\shco$, $\shco_r$ and $\spshco$ with each other. More concretely, we show in Proposition~\ref{prop:rel_sh} the following
\[
\shco(W)=\shco_0(W)\leq\shco_r(W)\leq\shco_{r'}(W)\leq\shco_2(W)=\spshco(W)
\]
for any closed $4$-manifold $W$ and $r,r'\in\R$ with $0\leq r<r'$. It is remarkable that $\shco_r$ is finite-to-one invariant if $r>0$, 
which will be shown in Proposition~\ref{prop:finite_to_one}. 
% \Erase{Note that $\shco$ is not so, and $\spshco$ is so.}
Note that $\spshco$ is also finite-to-one, but neither is $\shco$. 

The minimum of $r$ satisfying the inequality in Theorem~\ref{thm:complexity_genus} is $1/2$ (cf. Remark~\ref{rmk:best_result}), 
so we then focus on the behavior of $\shco_{1/2}$. 
Note that $\shco_{1/2}$ takes values in non-negative half integers. 
In Proposition~\ref{prop:Examples}, we determine the exact values of $\shco_{1/2}$ for infinitely many closed $4$-manifolds by using Theorem~\ref{thm:complexity_genus}. 
We also give the classification of all the $4$-manifolds with $1/2$-weighted shadow-complexity at most $1/2$. 
\begin{citingtheorem}
[\ref{thm:complexity0}]
The $1/2$-weighted shadow-complexity of a closed $4$-manifold $W$ is $0$ 
if and only if $W$ is diffeomorphic to either one of 
$S^4$, $\CP$, $\mCP$, $S^2\times S^2$, $2\CP$, $\CP\#\mCP$ or $2\mCP$. 
% \[
% S^4,\ \CP,\ \mCP,\ S^2\times S^2,\ \CP\#\CP,\ \CP\#\mCP\text{ or }\mCP\#\mCP. 
% \]
\end{citingtheorem}

\begin{citingtheorem}
[\ref{thm:complexity1/2}]
The $1/2$-weighted shadow-complexity of a closed $4$-manifold $W$ is $1/2$ 
if and only if 
$W$ is diffeomorphic to either one of $3\CP$, $2\CP\#\mCP$, $\CP\#2\mCP$, $3\mCP$, 
$S^1\times S^3$, $(S^1\times S^3)\#\CP$, $(S^1\times S^3)\#\mCP$, $\mathcal{S}_2$, ${\mathcal{S}}'_2$ or $\mathcal{S}_3$. 
\end{citingtheorem}
% In \cite{Mar05}, Martelli studied a certain invariant, which will be denoted by $w$ here, 
% defined by introducing a complexity called the \textit{weight} of Kirby diagrams. 
% Then he showed the following inequalities
% \[
% \frac13\spshco(W)\leq w(W) \leq 9 \spshco(W)+8 
% \]
% for any closed oriented smooth $4$-manifolds $W$, 
% where $\spshco$ is the special shadow-complexity. 
% By an easy discussion, we obtain 
% \[
% g(W)\leq w(W). 
% \]

\subsection*{Acknowledgement}
The first author was supported by JSPS KAKENHI Grant Number JP20K14316 and JSPS-VAST Joint Research Program Grant number JPJSBP120219602. 
%==================================================================================
\section{Preliminaries}
\subsection{Assumption and notations}
\begin{itemize}
 \item 
Any manifold is supposed to be compact, connected, oriented and smooth unless otherwise mentioned. 
 \item 
For triangulable spaces $A\subset B$, 
let $\Nbd(A;B)$ denote a regular neighborhood of $A$ in $B$. 
 \item 
For an $n$-manifold $W$ with $\partial W=\emptyset$ (resp. $\partial W\ne\emptyset$) 
and an integer $k$, 
we will use the notation $kW$ for 
the connected sum (resp. boundary connected sum) of $k$ copies of $W$ if $k>0$, 
for $S^n$ (resp. $B^n$) if $k=0$, 
and for the connected sum (resp. boundary connected sum) of $|k|$ copies of $W$ with the opposite orientation if $k<0$. 
 \item 
Let $\Sigma_{g,b}$ denote a compact surface of genus $g$ with $b$ boundary components. If $b=0$, we will write it as $\Sigma_g$ simply. 
\end{itemize}
% - - - - - - - - - - - - - - - - - - - - - - - - - - - - - - - - - - - - - - - - -
\subsection{Simple polyhedra and shadows}
Let $X$ be a connected compact space. 
We call $X$ a \textit{simple polyhedron} 
if a regular neighborhood $\Nbd(x;X)$ of each point $x\in X$ 
is homeomorphic to one of (i)-(iv) shown in Figure~\ref{fig:local_model}. 
% - - - - - - - - - - - - - - - - -
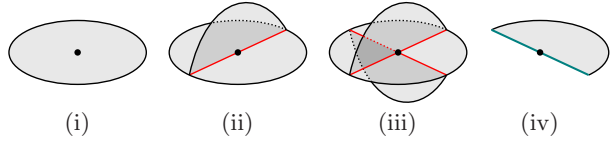
\begin{figure}[tbp]
\labellist
\footnotesize\hair 2pt
\pinlabel (i) [B] at    48.19 8.73
\pinlabel (ii) [B] at  147.40 8.73
\pinlabel (iii) [B] at 246.61 8.73
\pinlabel (iv) [B] at  334.49 8.73
\endlabellist
\centering
\includegraphics[width=.65\hsize]{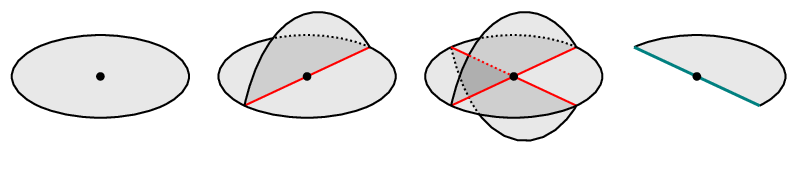}
\caption{Local models of simple polyhedra.}
\label{fig:local_model}
\end{figure}
% - - - - - - - - - - - - - - - - -
A point of type~(iii) is called a \textit{true vertex}. 
The set of all points of types (ii) and (iii) is called the \textit{singular set} of $X$, 
and it is denoted by $S(X)$. 
Note that $S(X)$ is disjoint union of circles and quartic graphs. 
A connected component of $S(X)$ with the true vertices removed is called a \textit{triple line}. 
Each connected component of $X\setminus S(X)$ is called a \textit{region}, 
and hence a region is homeomorphic to some surface. 
If $X$ has only disk regions, then $X$ is said to be \textit{special}. 
The set of points of type~(iv) is the \textit{boundary} of $X$, 
which is denoted by $\partial X$. 
If $\partial X$ is empty, the simple polyhedron $X$ is said to be \textit{closed}. 
If a region does not intersect $\partial X$, 
it is called an \textit{internal region}, and otherwise a \textit{boundary region}. 

Before defining shadows of $4$-manifold, 
we note that a simple polyhedron $X$ embedded in a $4$-manifold $W$ is said to be \textit{locally-flat} 
if a neighborhood $\Nbd(x;X)$ of each point $x\in X$ is contained in a smooth $3$-ball in $W$. 
\begin{definition}
\label{def:shadow}
% \begin{enumerate}
%  \item 
% A simple polyhedron $X$ embedded in a $4$-manifold $M$ with non-empty boundary is a \textit{shadow} 
% of $M$ if $M$ collapses onto $X$ and $X\cap \partial M=\partial X$. 
%  \item 
A simple polyhedron $X$ embedded in a closed $4$-manifold $W$ local-flatly is a \textit{shadow} of $W$ 
if $W\setminus\Int\Nbd(X;W)$ is diffeomorphic to $k(S^1\times B^3)$ for some $k\in\Z_{\geq0}$. 
% there exists an embedded graph $G$ such that $W\setminus \Int\Nbd(G;W)$ collapses onto $X$. 
%  \item 
% A simple polyhedron $X$ is a \textit{shadow} of a closed $3$-manifold $N$ 
% if $X$ is a shadow of a $4$-manifold whose boundary is diffeomorphic to $N$. 
% \end{enumerate}
\end{definition}
The notion of shadows was introduced by Turaev, who showed the following. 
\begin{theorem}[Turaev \cite{Tur94}]
\label{thm:cloed_4-manifold_shadow}
Any closed $4$-manifold admits a shadow. 
\end{theorem}

The \textit{complexity} of a simple polyhedron $X$ is the number of true vertices of $X$. 
Theorem~\ref{thm:cloed_4-manifold_shadow} allows us to define an invariant of closed $4$-manifolds like the Matveev complexity of $3$-manifolds. 
\begin{definition}

Let $W$ be a closed $4$-manifold. 
The \textit{shadow-complexity} $\shco(W)$ of $W$ is defined as the minimum of the complexities over all shadows of $W$.
The \textit{special shadow-complexity} $\spshco(W)$ of $W$ is defined as the minimum of the complexities over all special shadows of $W$.
\end{definition}
This notion was introduced by Costantino in \cite{Cos06b}. 
See \cite{CT08,IK17,KMN18,Mar11} for the studies regarding the (special) shadow-complexity. 
 
\subsection{Gleams and shadowed polyhedra}
We then define the $\Z_2$-\textit{gleam} of a simple polyhedron $X$. 
Let $R$ be an internal region of $X$. 
Then $R$ is homeomorphic to the interior of some compact surface $F$, 
and the homeomorphism $\Int F\to R$ will be denoted by $f$. 
This $f$ can extend to a local homeomorphism $\overline{f}:F\to X$. 
Moreover, there exists a simple polyhedron $\widetilde{F}$ obtained from $F$ 
by attaching an annulus or a M\"obius band to each boundary component of $F$ along the core circle 
such that $\overline{f}$ can extend to a local homeomorphism $\widetilde{f}:\widetilde{F}\to X$. 
Then the number of the M\"obius bands attached to $F$ modulo $2$ 
is called the $\Z_2$-\textit{gleam} of $R$ and is denoted by $\gl_2(R)\in\{0,1\}$. 
% Note that this number depends only on the combinatorial data of $X$. 
Note that this number is determined only by $X$ combinatorially. 
\begin{definition}
A \textit{gleam function}, or simply \textit{gleam}, of a simple polyhedron $X$ is 
a function associating to each internal region $R$ of $X$ 
a half-integer $\gl(R)$ satisfying $\gl(R)+\frac12\gl_2(R)\in\Z$. 
The value $\gl(R)$ is called the \textit{gleam} of $R$. 
A simple polyhedron equipped with a gleam is called a \textit{shadowed polyhedron}. 
\end{definition}

\begin{theorem}[Turaev \cite{Tur94}]
There exists a canonical way to associate to a shadowed polyhedron $X$ a $4$-manifold $M_X$ with boundary such that 
\begin{itemize}
 \item
$X$ is local-flatly embedded in $M_X$, 
 \item
$M_X$ collapses onto $X$, and 
 \item 
$X\cap\partial M_X=\partial X$. 
\end{itemize}
\end{theorem}
\begin{remark}
The polyhedron $X$ is also called a shadow of $M_X$, 
and the $4$-manifold $M_X$ with boundary is often called the $4$-dimensional thickening of $X$. 
\end{remark}

For a shadowed polyhedron $X$, if $\partial M_X$ is diffeomorphic to $k(S^1\times S^2)$, 
one can obtain a closed $4$-manifold $W$ by gluing $k(S^1\times B^3)$ to $M_X$ along their boundaries. 
It is easy to see that $X$ is embedded in the $4$-manifold $W$ as a shadow. 
Due to Laudenbach and Po\'enaru \cite{LP72}, $W$ is uniquely determined up to diffeomorphism, 
and hence shadowed polyhedra can be treated as a description of closed $4$-manifolds. 

Conversely, if a shadow $X$ of a closed $4$-manifold $W$ is given, 
there exists a canonical way to compute a gleam of $X$ 
such that the obtained shadowed polyhedron describes the $4$-manifold $W$ in the above sense.  
% A shadowed polyhedron obtained in such a way actually describes the $4$-manifold $W$. 
Here we review how to compute the gleam below. 
Let $W$ be a closed $4$-manifold and $X$ a shadow of $W$. 
Let $R$ be an internal region of $X$, 
and hence the boundary of $R$ (as a topological space) is contained in $S(X)$. 
Set $X_S=\Nbd(S(X);X)$ and $\bar{R}=R\setminus\Int X_S$. 
As shown in \cite{Mar05}, 
there exists a $3$-manifold $N_S$ with boundary satisfying 
\begin{itemize}
 \item
$N_S$ is smoothly embedded in $W$, 
 \item 
$N_S\cap X=X_S$, and 
 \item 
$N_S$ collapses onto $X_S$. 
\end{itemize}
Note that $N_S$ is homeomorphic to the disjoint union of some $3$-dimensional handlebodies that are possibly non-orientable. 
Set $I_R=\Nbd(\partial \bar{R};\partial N_S)$, 
which can be seen as an interval-bundle over $\partial \bar{R}$. 
Thus, $I_R$ is the disjoint union of some annuli and M\"obius bands. 
Let $\bar{R}'$ be a small perturbation of $\bar{R}$ such that $\partial \bar{R}'\subset I_R$, 
and we can assume that $\bar{R}$ and $\bar{R}'$ intersect transversely at a finite number of points. 
Then the gleam we require is given by 
\[
\gl(R)=\#(\Int \bar{R}\cap\Int \bar{R}')+\frac12\#(\partial \bar{R}\cap\partial \bar{R}'), 
\]
where the intersections are counted with signs. 
% - - - - - - - - - - - - - - - - - - - - - - - - - - - - - - - - - - - - - - - - -
\subsection{Encoding graph}
\label{subsec:Encoding_graph}
In this subsection, we review an encoding graph that is a graph describing a simple polyhedron without true vertices. 
Set 
\[
Y=\left\{z\in\C\mid \arg z \in\{0,\ 2\pi/3,\ 4\pi/3\},|z|\leq1 \right\}\cup\{0\}, 
\]
and let $f_{111}$, $f_{12}$ and $f_{3}$ be self-homeomorphisms on $Y$ 
that send $z$ to, respectively, $z$, $\bar z$ and $e^{2\pi\sqrt{-1}/3}z$. 
Then, for $\sigma\in\{111,12,3\}$, 
let $Y_{\sigma}$ denote the mapping torus of $f_{\sigma}$. 
Note that the numbers of boundary components of $f_{111}$, $f_{12}$ and $f_{3}$ are 
$3$, $2$ and $1$, respectively. 
It is easy to see that if a simple polyhedron has a circle component in the singular set, 
its regular neighborhood is homeomorphic to either one of $Y_{111}$, $Y_{12}$ or $Y_3$. 

Let $X$ be a simple polyhedron with no true vertices. 
Since a connected component of $S(X)$ is homeomorphic to $S^1$, 
$X$ is decomposed into a finite number of $Y_{111}$, $Y_{12}$, $Y_3$, 
a $2$-disk $D$, a pair of pants $P$ and a M\"obius band $Y_2$. 
Such a decomposition of $X$ induces a graph consisting of vertices as shown in 
Figure~\ref{fig:encoding_graph} corresponding to the pieces in the decomposition or boundary components. 
An edge is associated to each circle along which $X$ is decomposed. 
The graph obtained in such a way is called an \textit{encoding graph} of $X$. 
% - - - - - - - - - - - - - - - - -
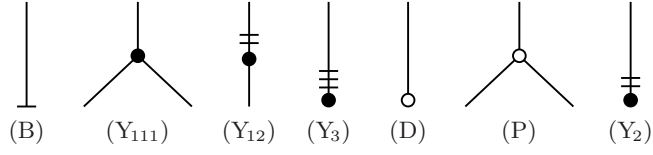
\begin{figure}[tbp]
\labellist
\footnotesize\hair 2pt
\pinlabel $(\mathrm{B})$         [B] at  13.23 8.12
\pinlabel $(\mathrm{Y}\!_{111})$ [B] at  62.83 8.12
\pinlabel $(\mathrm{Y}\!_{12})$  [B] at 112.44 8.12
\pinlabel $(\mathrm{Y}\!_{3})$   [B] at 147.87 8.12
\pinlabel $(\mathrm{D})$         [B] at 183.31 8.12
\pinlabel $(\mathrm{P})$         [B] at 232.91 8.12
\pinlabel $(\mathrm{Y}\!_{2})$   [B] at 282.52 8.12
\endlabellist
\centering
\includegraphics[width=.7\hsize]{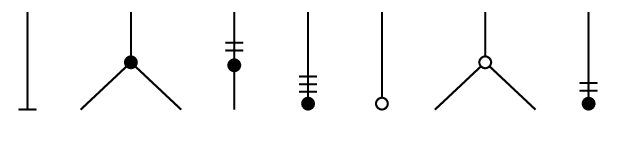}
\caption{Vertices of an encoding graph.}
\label{fig:encoding_graph}
\end{figure}
% - - - - - - - - - - - - - - - - -
See \cite{Mar11,KMN18} for more details. 

Let $G$ be an encoding graph of a simple polyhedron $X$ with no true vertices. 
As mentioned in \cite{Mar11,KMN18}, $X$ can not be recovered only by $G$ if $G$ has a cycle 
since the mapping class group of $S^1$ is $\Z/2\Z$. 
Actually, a pair of $G$ and a cocycle $\alpha\in H^1(G;\Z/2\Z)$ can determine $X$ (here we omit the details of how they do it). 

% In general, there exists an embedding $e:G\to X$ such that $e(G)$ is a retraction of $X$. 
% Moreover, we can assume that $e(G)\cap S(X)$ coincides with the images of 
% vertices of types $(\mathrm{Y}\!_{111})$, $(\mathrm{Y}\!_{12})$ and $(\mathrm{Y}\!_{3})$. 
% Especially, the images of edges of $G$ does not intersect $S(X)$. 
% Suppose $G$ has a non-trivial loop $\gamma$. 
% Then either one of an annulus or a M\"obius band can be embedded in $X$ along $e(\gamma)$. 
% In the former case, $\alpha([\gamma])$ must be $0$, and in the latter case $\alpha([\gamma])=1$, 
% which is the meaning of $\alpha$. 

% - - - - - - - - - - - - - - - - - - - - - - - - - - - - - - - - - - - - - - - - -
\subsection{Trisections}
\label{subsec:Trisections}
Here we review the notion of trisections of closed $4$-manifolds. 
\begin{definition}
\label{def:trisection}
Let $W$ be a closed $4$-manifold and $g,k_1,k_2,k_3$ non-negative integers with $\max\{k_1,k_2,k_3\}\leq g$. 
A \textit{$(g; k_1, k_2, k_3)$-trisection}, or simply a \textit{trisection}, of $W$ is a data of a decomposition of $W$ 
into three submanifolds $W_1,W_2$ and $W_3$ 
such that the following three conditions hold;
\begin{itemize}
 \item 
for $i\in\{1,2,3\}$, $W_i$ is diffeomorphic to ${k_i} (S^1\times B^3)$, 
 \item
for $i,j\in\{1,2,3\}$ with $i\ne j$, the intersection $H_{ij}=W_i\cap W_j$ 
is diffeomorphic to a genus $g$ $3$-dimensional handlebody $g (S^1\times D^2)$, and 
 \item 
the intersection $W_1\cap W_2\cap W_3$ is diffeomorphic to $\Sigma_g$. 
\end{itemize}
The surface $W_1\cap W_1\cap W_2$ is called the \textit{central surface} of the trisection. 
The \textit{genus} of a trisection is the genus of its central surface. 
\end{definition}
This notion was introduced by Gay and Kirby \cite{GK16}, and they showed the following by using a certain generic map from $4$-manifolds to the plane $\R^2$. 
\begin{theorem}[Gay and Kirby \cite{GK16}]
\label{thm:trisection}
Any closed $4$-manifold admits a trisection.
\end{theorem}

A \textit{trisection diagram} of the trisection $W_1\cup W_2\cup W_3$ is a 4-tuple $(\Sigma_g, \alpha, \beta, \gamma)$ 
such that $\Sigma_g$ is the central surface and that $\alpha$, $\beta$, and $\gamma$ are cut systems of $H_{31}$, $H_{12}$, and $H_{23}$, respectively. 
Here a \textit{cut system} of a $3$-dimensional handlebody $H$ 
is a collection of the boundaries of properly embedded disks in $H$ such that they cut $H$ open into a single $3$-ball. 
We note that $\partial W_i$ is decomposed into $H_{ij}\cup H_{ik}$ for $\{i, j, k\}=\{1, 2, 3\}$, 
which is a genus $g$ Heegaard splitting of $\partial X_i$ since $H_{ij}\cap H_{ik} = \partial H_{ij}=\partial H_{ik}$. 
Therefore, $(\Sigma_g, \alpha, \beta)$, $(\Sigma_g, \beta, \gamma)$ and $(\Sigma_g, \gamma, \alpha)$ are Heegaard diagrams of 
$\partial W_1$, $\partial W_2$ and $\partial W_3$, respectively. 
We also note that a trisection diagram reconstructs the corresponding $4$-manifolds and the trisection uniquely up to diffeomorphisms \cite{GK16}. 

We here define an operation called a \textit{stabilization} for a trisection diagram $(\Sigma_g, \alpha, \beta, \gamma)$. 
It is obtained by connected summing $(\Sigma_g, \alpha, \beta, \gamma)$ with either one of the diagrams shown in Figure~\ref{fig:destab}.
% - - - - - - - - - - - - - - - - -
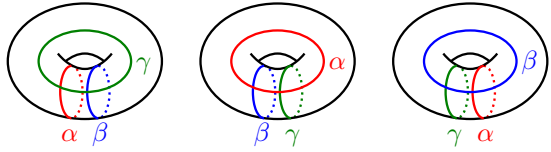
\begin{figure}[tbp]
\labellist
\small\hair 2pt
\pinlabel {\textcolor[rgb]{1, 0, 0}{$\alpha$}}       [B]  at  28.35 10.73
\pinlabel {\textcolor[rgb]{0, 0, 1}{$\beta$}}        [B]  at  39.69 10.73
\pinlabel {\textcolor[rgb]{0, 0.50196, 0}{$\gamma$}} [Bl] at  51.01 37.08
\pinlabel {\textcolor[rgb]{0, 0, 1}{$\beta$}}        [B]  at  99.21 10.73
\pinlabel {\textcolor[rgb]{0, 0.50196, 0}{$\gamma$}} [B]  at 110.55 10.73
\pinlabel {\textcolor[rgb]{1, 0, 0}{$\alpha$}}       [Bl] at 121.87 37.08
\pinlabel {\textcolor[rgb]{0, 0.50196, 0}{$\gamma$}} [B]  at 170.08 10.73
\pinlabel {\textcolor[rgb]{1, 0, 0}{$\alpha$}}       [B]  at 181.42 10.73
\pinlabel {\textcolor[rgb]{0, 0, 1}{$\beta$}}        [Bl] at 192.74 37.08
\endlabellist
\centering
\includegraphics[width=.6\hsize]{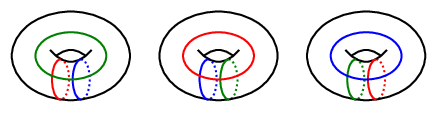}
\caption{A destabilization triple.}
\label{fig:destab}
\end{figure}
% - - - - - - - - - - - - - - - - -
By this operation, the corresponding $4$-manifold does not change up to diffeomorphisms, 
and the genus of the corresponding trisection increases by $1$. 
We also define an operation called a \textit{destabilization} as the inverse of a stabilization. 
Note that any two trisection diagrams of the same $4$-manifold are related by stabilizations, destabilizations and diffeomorphisms \cite{GK16,MSZ16}. 

Let $(\Sigma_g,\alpha,\beta,\gamma)$ be a trisection diagram. 
We note that each of $\alpha,\beta$ and $\gamma$ consists of $g$ mutually disjoint simple closed curves, 
so we will write $\alpha=\alpha_1\sqcup\cdots\sqcup\alpha_g,\ \beta=\beta_1\sqcup\cdots\sqcup\beta_g$, and $\gamma=\gamma_1\sqcup\cdots\sqcup\gamma_g$. 
Suppose that there exist $h,i,j\in\{1,\ldots,g\}$ such that 
\begin{itemize}
 \item 
% exactly one of $\alpha_h, \beta_i$ and $\gamma_j$ intersects each of the other two transversely once, and 
exactly two of $\alpha_h, \beta_i$ and $\gamma_j$ are parallel, and 
 \item 
% the other two are parallel. 
each of the parallel two curves intersects the other one transversely exactly once. 
\end{itemize}
We call such a triple $(\alpha_h, \beta_i, \gamma_j)$ a \textit{destabilization triple}. 
By handle sliding certain curves over $\alpha_h, \beta_i$ and $\gamma_j$ if necessary, 
we can assume that $\alpha_h, \beta_i$ and $\gamma_j$ do not intersect $\alpha\cup\beta\cup\gamma\setminus(\alpha_h\cup\beta_i\cup\gamma_j)$. 
Especially, the union of $\alpha_h, \beta_i$ and $\gamma_j$ is contained in a punctured torus after this modification, 
which allows the trisection diagram to be destabilized once. 
% Thus, if we just find a destabilization triple, then the trisection diagram once. 

We close this subsection with the definition of the trisection genus of closed $4$-manifolds. 
\begin{definition}
Let $W$ be a closed $4$-manifold. 
The \textit{trisection genus} $g(W)$ of $W$ is defined as the minimal genus of any trisection of $W$. 
\end{definition}
It is obvious that the trisection genus is an invariant of closed $4$-manifolds that takes a value in $\Z_{\geq0}$. 

\subsection{Handle decompositions to trisections}
\label{subsec:hdl_dec_to_tris}
Meier and Zupan showed the existence of a bridge trisection for any knotted surface 
by constructing a trisection from a handle decomposition of the ambient $4$-manifold \cite{MZ18}. 
Here we review their method to construct a trisection. 

Let $W$ be a closed $4$-manifold, 
and let us give a handle decomposition of $W$ such that each handle is attached to those with lower indices. 
Suppose that it has at least one $2$-handle and exactly one each of $0$-handle and $4$-handle. 
Let $H^i$ denote the union of all the $i$-handles 
and $L\subset \partial (H^0\cup H^1)$ the attaching link of the $2$-handles. 
Let $\tau$ be an unknotting tunnel for $L$ in $\partial(H^0\cup H^1)$, 
which means that $\partial \Nbd(L\cup\tau; \partial (H^0\cup H^1))$ gives a Heegaard splitting of 
$\partial (H^0\cup H^1)$. 
Set $\Sigma=\partial \Nbd(L\cup\tau; \partial (H^0\cup H^1))$. 
Then $W$ is trisected by 
\begin{align*}
W_1 &= (H^0\cup H^1) \setminus \Int \Nbd(L\cup\tau; W),\\
W_2 &= H^2 \cup \Nbd(L\cup\tau; W) \text{ and}\\
W_3 &= (H^3\cup H^4) \setminus \Int \Nbd(L\cup\tau; W) 
\end{align*}
with central surface $\Sigma$. 

% \Red
% The \textit{tunnel number} $t(L)$ is the minimum number of components of unknotting tunnel of a link $L$. 
% If $L$ is in the $3$-sphere, it is immediate to see that $t(L)$ is bounded by the crossing number $c(L)$ from above. 
% In the case where $L$ is in $k(S^1\times S^2)$ for some $k\in\Z_{\geq1}$, 
% describing $L$ as a diagram in 
% \Black

A trisection diagram for the trisection obtained above is given by letting 
$\alpha$, $\beta$ and $\gamma$ be 
cut systems of $\partial (H^0\cup H^1)\setminus \Int \Nbd(L\cup\tau; \partial (H^0\cup H^1))$, $\Nbd(L\cup\tau; \partial (H^0\cup H^1))$ and $\Nbd(L\cup\tau;\partial(H^3\cup H^4))$, respectively. 
More concretely, we can describe $\beta$ and $\gamma$ as follows. 
Let $\tau_1,\ldots,\tau_n$ be the connected components of $\tau$, 
and suppose that $L\cup(\tau_1\sqcup\cdots\sqcup\tau_{\ell-1})$ is connected, 
where $\ell$ is the number of components of $L$. 
We consider the framings of $L$ as a link $L'$ parallel to $L$, 
and we suppose that $L'$ lies on $\Sigma=\partial \Nbd(L\cup\tau; \partial (H^0\cup H^1))$. 
Then, $\beta$ is given as meridians of $L$ and those of $\tau_{\ell}\sqcup\cdots\sqcup\tau_n$, 
and $\gamma$ is given as $L'$ and meridians of $\tau_{\ell}\sqcup\cdots\sqcup\tau_n$. 

% - - - - - - - - - - - - - - - - -
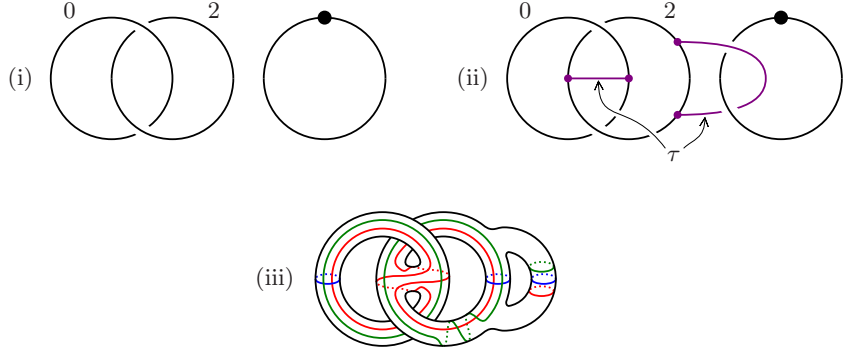
\begin{figure}[tbp]
\labellist
\footnotesize\hair 2pt
\pinlabel {$0$}    [Br] at 251.49 158.74
\pinlabel {$2$}    [Bl] at  308.18 158.74
\pinlabel {$\tau$} [B]  at  315.27  92.13
\pinlabel {$0$}    [Br] at  38.89 158.74
\pinlabel {$2$}    [Bl] at   95.58 158.74
\pinlabel {(i)}    [Br] at  19.05 127.56
\pinlabel {(ii)}   [Br] at 231.65 127.56
\pinlabel {(iii)}  [Br] at 140.94  34.02
\endlabellist
\centering
\includegraphics[width=.9\hsize]{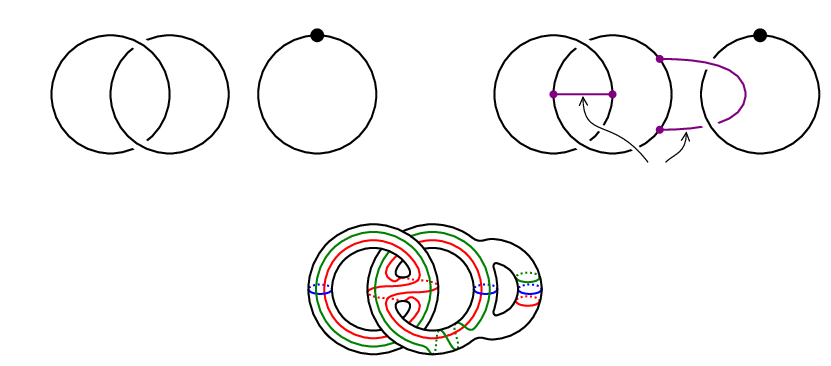}
\caption{%$A trisection obtained from a handle decomposition.
(i) A Kirby diagram of $(S^2 \times S^2) \# (S^1 \times S^3)$. (ii) An unknotting tunnel for the attaching link of the $2$-handles. (iii) A trisection diagram of $(S^2 \times S^2) \# (S^1 \times S^3)$. }
\label{fig:example_tris}
\end{figure}
% - - - - - - - - - - - - - - - - -
See Figure~\ref{fig:example_tris} for an example. 
The Kirby diagram depicted in (i) represents $(S^2\times S^2)\#(S^1\times S^3)$, 
where the attaching link $L$ is given as a Hopf link in $S^1\times S^2=\partial(S^1\times B^3)$. 
We can find an unknotting tunnel for $L$ such as in (ii). 
Then the trisection obtained from them in the way explained in this subsection is represented by the diagram shown in (iii).
%==================================================================================
\section{Cut systems and weighted complexity}
In this section, we introduce a new complexity called the \textit{weighted complexity} $c_r$, 
and by using it, we define the \textit{weighted shadow-complexity} $\shco_r$ of closed $4$-manifolds. 
After the definitions, we discuss some properties of $\shco_r$, especially, 
relationships with %already known complexities; 
the shadow-complexity and the special shadow-complexity. 

% - - - - - - - - - - - - - - - - - - - - - - - - - - - - - - - - - - - - - - - - -
Let $X$ be a simple polyhedron with $S(X)\ne\emptyset$. 
We define a \textit{cut system} for $X$ as a collection $\Gamma$ of mutually disjoint arcs % $\gamma_1,\ldots,\gamma_l$
embedded in $X$ such that 
\begin{itemize}
\item
each endpoint of the arcs lies in a triple line or $\partial X$, 
\item
the interiors of the arcs are contained in $X\setminus(S(X)\cup\partial X)$,
\item 
each component of $\partial X$ intersects exactly one arc, and
\item
each region with $\Gamma$ removed is simply connected.
\end{itemize}
Therefore, 
$\Gamma$ can be understood as a collection of cocores of $1$-handles of some handle decomposition of the regions. 
Note that $S(X) \cup\Gamma\cup \partial X$ is connected even if $S(X)$ is not connected. 
It is easy to see that the number of arcs of $\Gamma$ lying in a region $R$ is exactly $1-\chi(R)$. 
% For a simple polyhedron $X$ with $S(X)\cup\partial X=\emptyset$ (namely, a closed surface), 
% a \textit{cut system} for $X$ is defined as an embedded trivalent graph in $X$ whose complement in $X$ is an open disk. 

Recall that the \textit{complexity} of a simple polyhedron $X$ is defined as the number $c(X)$ of true vertices of $X$, 
which of course depends only on the shape of the singular set. 
We here introduce a new complexity to take into consideration the ``non-trivialities'' of regions. 
\begin{definition}
Fix a real number $r\geq0$. 
The \textit{$r$-weighted complexity} $c_r(X)$ of a simple polyhedron $X$ 
is defined as 
\[
c_r(X)=c(X)+\sum_{R:\text{region}}r(1-\chi(R)) 
\]
if $X$ is not a closed surface, and set $c_r(X)=0$ if $X$ is homeomorphic to $S^2$. 
The \textit{$r$-weighted shadow-complexity} $\shco_r(W)$ of a $4$-manifold $W$ is defined as 
the minimum of the $r$-weighted complexities over all shadows of $W$. 
\end{definition}
% One might wonder why $c_r$ is not defined for closed surfaces exccept for $S^2$. 
% The reason is in Lemma~\ref{lem:closed_surf}. 
We will show in Lemma~\ref{lem:closed_surf} that any closed surface except for $S^2$ can not be a shadow of any closed $4$-manifold, 
which is the reason why we do not define $c_r$ for closed surfaces except for $S^2$. 

Note that $c_0(X)=c(X)$ and $c_r(X)\leq c_{r'}(X)$ if $r<r'$. 
We show important relationships between the weighted shadow-complexity, 
the shadow-complexity and the special shadow-complexity. 
\begin{proposition}
\label{prop:rel_sh}
Let $W$ be a closed $4$-manifold and $r,r'\in\R$.  
\begin{enumerate}
 \item
If $0<r<r'$, then the following hold:
\[
\shco(W)\leq \shco_r(W)\leq \shco_{r'}(W)\leq\spshco(W). 
\]
 \item
$\shco(W)= \shco_0(W)$. 
 \item
$\shco_r(W)=\spshco(W)$ if $r\geq2$. 
\end{enumerate}

\end{proposition}
\begin{proof}
(1) Obviously, $c(X)\leq c_r(X)$ for a simple polyhedron $X$, 
and hence the first inequality $\shco(W)\leq \shco_r(W)$ holds. 
If a simple polyhedron $X$ is special, then $c(X) = c_r(X)$. 
Therefore, $\shco_r(W)\leq\spshco(W)$ holds. \\[2mm]
\noindent (2) 
It is obvious from the definition of $r$-weighted complexity. \\[2mm]
\noindent (3)
% Suppose $s\geq2$, and we prove $\shco_r(W)=\spshco(W)$. 
Let $X$ be a shadow of $W$. 
It is enough to check that $\spshco(W)\leq c_r(X)$. 
We first consider the case $S(X)=\emptyset$. 
We will show in Lemma~\ref{lem:closed_surf} 
that a closed surface of non-zero genus can not be a shadow of any closed $4$-manifold. 
Thus, $X$ must be homeomorphic to $S^2$ or has non-empty boundary. 
If $X$ is homeomorphic to $S^2$, then $W$ is diffeomorphic to $S^4$, $\CP$ or $\mCP$. 
Then $\spshco(W)=0< 2\leq r=c_r(X)$ holds. 
If $X$ has non-empty boundary, 
then $W$ is diffeomorphic to $k(S^1\times S^3)$, where $k=1-\chi(X)=\frac{c_r(X)}r$. 
% \sout{By the assumption, $k=0$, that is, $W$ is $S^4$, 
% and $\spshco(W)\leq c_r(X)$ also holds.} 
If $k=0$, that is, $W$ is $S^4$, 
then $\spshco(W)\leq c_r(X)$ also holds. 
Suppose $k\geq1$. 
As shown in \cite{Nao23}, the special shadow-complexity of $k(S^1\times S^3)$ is equal to $k+1$. 
Thus, we also have $\spshco(W)=k+1\leq 2k\leq rk=c_r(X)$. 

We next consider the case $S(X)\ne\emptyset$. 
Let $\Gamma$ be a cut system for $X$. 
Recall that $\sum_{R:\text{region}}r(1-\chi(R))$ 
is equal to the number of arcs of $\Gamma$. 
Let $e$ be one of arcs of $\Gamma$. 
Then $\Nbd(e;X)$ is shown in the leftmost part of Figure~\ref{fig:specialization}-(i) if 
both of the endpoints of $e$ lie in $S(X)$, 
and otherwise $\Nbd(e\cup C;X)$ is shown in the leftmost part of Figure~\ref{fig:specialization}-(ii), 
where $C$ is the boundary component of $X$ containing an endpoint of $e$. 
The move shown in Figure~\ref{fig:specialization}-(i) is called a \textit{$(0\to2)$-move} (cf. \cite{Tur94,Cos05}), 
which creates two true vertices and decrease the number of arcs of $\Gamma$ by $1$. 
Figure~\ref{fig:specialization}-(ii) shows the composition of three moves. 
The first move (ii-1) is a \textit{$(0\to1)$-move} (cf. \cite{Tur94,Cos05}), 
and the second move (ii-2) is a $(0\to2)$-move. 
By these two moves, three true vertices and one annular boundary region are created. 
The move (ii-3) is a collapsing so that the annular boundary region is removed. 
By this collapsing, one true vertex is removed. 
The move (ii) that is the composition of (ii-1), (ii-2) and (ii-3) 
changes the simple polyhedron so that two true vertices are created 
and decrease the number of arcs of $\Gamma$ by $1$. 
We apply a move (i) or a move (ii) for every arc of $\Gamma$, 
so that we obtain a special polyhedron $X'$ with 
$c(X')=c(X)+2\sum_{R:\text{region}}(1-\chi(R))=c_2(X)\leq c_r(X)$. 
Therefore, we have $\spshco(W)\leq c_r(X)$. 
\end{proof}
% - - - - - - - - - - - - - - - - -
\begin{figure}[tbp]
\labellist
\footnotesize\hair 2pt
\pinlabel {\textcolor[rgb]{0, 0.50196, 0.50196}{$\partial X$}} [Bl] at  164.72 215.66
\pinlabel $1/2$    [B]  at  340.47 225.58
\pinlabel $1/2$    [B]  at  306.45 225.58
\pinlabel -$1$     [B]  at  263.93 209.99
\pinlabel $1/2$    [B]  at  280.94 48.42
\pinlabel $0$      [B]  at  269.60 72.51
\pinlabel $1/2$    [B]  at  235.59 69.68
\pinlabel -$1$     [B]  at  193.07 54.09
\pinlabel glue     [Bl] at   91.02 164.64
\pinlabel glue     [Bl] at  303.62 164.64
\pinlabel glue     [Bl] at  232.75 8.73
\pinlabel glue     [Bl] at   91.02 306.37
\pinlabel glue     [Bl] at  346.14 306.37
\pinlabel glue     [Br] at  176.06 425.42

\pinlabel glue     [Br] at  431.18 425.42
\pinlabel $0$      [B]  at  391.49 367.31

\pinlabel (i)      [Br] at  22.99 365.90
\pinlabel (ii)     [Br] at  22.99 224.16
\pinlabel {arc $e$ of $\Gamma$} [Bl] at 150.55 297.86

\pinlabel $0$      [B]  at  482.20 72.51
\pinlabel $1/2$    [B]  at  448.19 69.68
\pinlabel -$1/2$    [B] at 493.54 48.42

\pinlabel glue     [Bl] at  445.35 8.73
\pinlabel (ii-1)   [B] at  221.41 238.34
\pinlabel (ii-2)   [Bl] at  320.63 136.29
\pinlabel (ii-3)   [B] at  363.15 82.43
\endlabellist
\centering
\includegraphics[width=1\hsize]{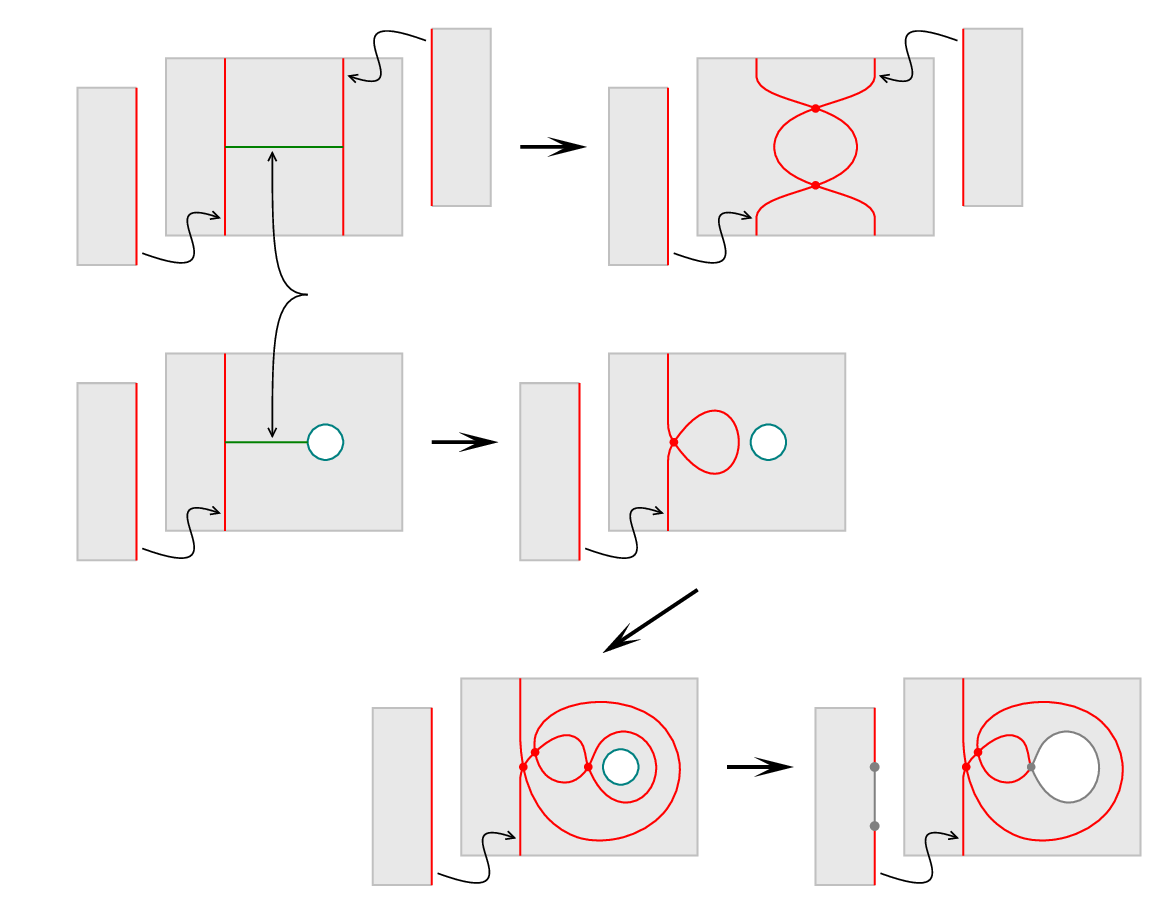}
\caption{Modification of a simple polyhedron into a special polyhedron.}
\label{fig:specialization}
\end{figure}
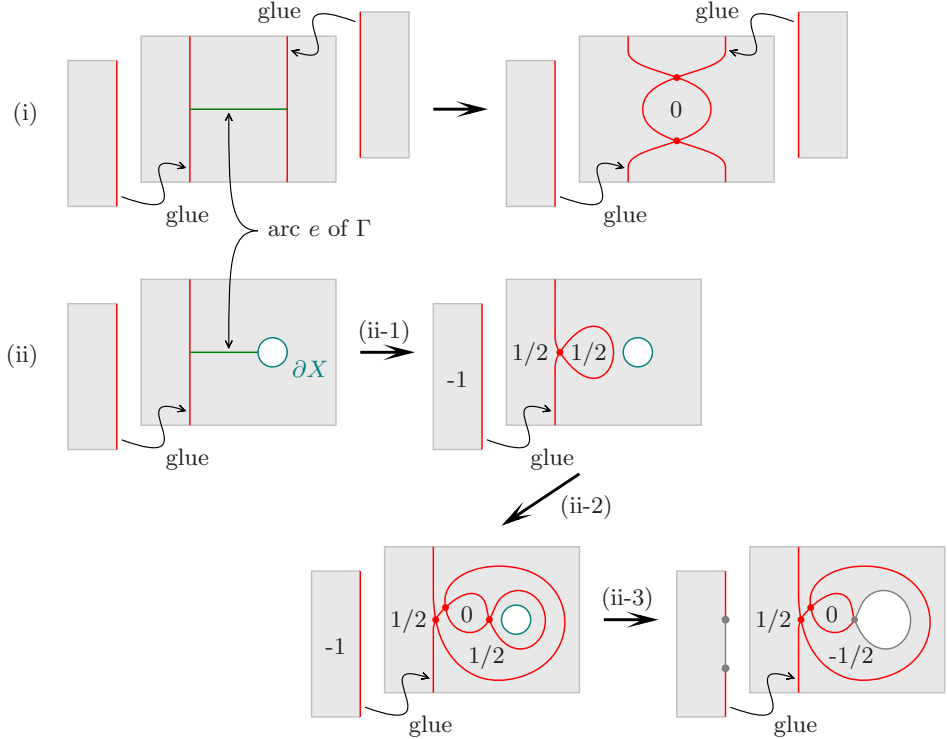
% - - - - - - - - - - - - - - - - -

Let $X$ and $X'$ be shadows of closed $4$-manifolds $W$ and $W'$, respectively. 
We choose small disks $D$ and $D'$ in regions of $X$ and $X'$, respectively. 
Identifying $D$ and $D'$, we obtain a new simple polyhedron $X''$. 
This polyhedron $X''$ is a shadow of the $4$-manifold $W\# W'$. 
By this operation, the summation of the Euler characteristics of the regions decreases by $2$, so we have the following. 
\begin{proposition}
\label{prop:connected_sum_ineq}
For any closed $4$-manifolds $W$ and $W'$, 
$\shco_r(W\# W')\leq\shco_r(W)+\shco_r(W')+2r$. 
\end{proposition}
% As a consequence of Theorems~\ref{thm:complexity0} and \ref{thm:complexity1/2}, 
% we have $\shco_{1/2}(4\CP)=1$. 
% Therefore, the inequality in Proposition~\ref{prop:connected_sum_ineq} is the best possible. 
% \Erase{More generally, we have $\shco_{1/2}(k\CP)=(k-2)/2$ for any integer $k\geq2$ (Proposition~\ref{prop:Examples}). }
As shown in Proposition~\ref{prop:Examples}, we have $\shco_{1/2}(2\CP)=0$ and $\shco_{1/2}(4\CP)=1$. 
These give an example satisfying the equality $\shco_r(W\# W')=\shco_r(W)+\shco_r(W')+2r$ as
$r=1/2,\ W=W'=2\CP$.

We then discuss the finiteness of the complexities. 
There exist infinitely many closed $4$-manifolds with shadow-complexity $0$. 
For example, $\shco(k\CP)=0$ for any $k\in\Z$ (cf. \cite{Mar11} and Proposition~\ref{prop:Examples}). 
Thus, the shadow-complexity for closed $4$-manifold is not finite-to-one. 
On the other hand, the special shadow-complexity is finite-to-one \cite{Cos06b,Mar05}. 
We here show that the weighted shadow-complexity is also finite-to-one. 
\begin{proposition}
\label{prop:finite_to_one}
% For each non-negative half-integer $r$, 
% there exists a finite number of closed $4$-manifolds 
% having cut shadow-complexity less than or equal to $r$.
For any positive number $r$ and any non-negative number $a$, 
there exists a finite number of closed $4$-manifolds 
having $r$-weighted shadow-complexity less than or equal to $a$. 
\end{proposition}
\begin{proof}
Fix $r>0$ and $a\geq0$. 
Note that the $r$-weighted complexity $c_r$ takes a value in 
$\{m+rn\mid m,n\in\Z_{\geq0}\}$. 
The set $\{m+rn\mid m,n\in\Z_{\geq0}\}\cap[0,a]$ is a finite set, 
in which we pick arbitrary $a_0$. 
The number of ways to present $a_0$ in a form $m+rn$ is finite, 
so fix $m_0,n_0\in\Z_{\geq0}$ with $a_0=m_0+rn_0$. 
It is easy to check that the number of simple polyhedra 
with $m_0$ true vertices and $\sum_{R:\text{region}}(1-\chi(R))=n_0$ is finite. 
By Martelli's result \cite[Theorem 2.4]{Mar05}, 
the number of closed $4$-manifolds admitting a shadow homeomorphic to a fixed simple polyhedron is finite. 
Therefore, the proposition holds. 
\end{proof}

%==================================================================================
\section{Kirby diagrams and trisections from shadows}
% - - - - - - - - - - - - - - - - - - - - - - - - - - - - - - - - - - - - - - - - -
In this section, we explain how one can draw a Kirby diagram of a 
$4$-manifold $W$ from a given shadow of $W$. 
We refer the reader to \cite{KN20} for the case of special shadows, 
and we stress that 
% we now consider shadows whose boundary may be non-empty and the regions may be disks. 
shadows we will consider can have non-empty boundary and non-disk regions. 
We also give the proof of Theorem~\ref{thm:complexity_genus} at the end of the section. 
\subsection{Shadows to Kirby diagrams}
Let $X$ be a shadow of a $4$-manifold $W$ with $S(X)\ne \emptyset$ 
and $\Gamma$ a cut system for $X$. 
Set $\tilde\Gamma = S(X)\cup\Gamma\cup\partial X$, which will be regarded as a graph naturally. 
% fix a maximal tree $T$ of $\tilde\Gamma$ 
% so that $T$ does not have an edge whose endpoints lie in the same triple line of $X$. 
% In other words, 
% if a triple line is subdivided into some edges in the graph $\tilde\Gamma$, 
% then $T$ must have all such edges except one. 
Let $T_0$ be a forest each of whose connected component is a spanning tree of a connected component of $S(X)$ as a subgraph of $\tilde\Gamma$. 
Then let $T$ be a spanning tree of $\tilde\Gamma$ obtained from $T_0$ by adding some edges of $\tilde\Gamma$. 

Set $X_{\tilde\Gamma}=\Nbd(\tilde\Gamma;X)$. 
The number of connected components of 
$\partial X_{\tilde\Gamma}\setminus\partial X$ is the same as that of the regions of $X$, 
and $X$ is obtained from $X_{\tilde\Gamma}$ by capping $\partial X_{\tilde\Gamma}\setminus\partial X$ off by $2$-disks. 
Especially, $X\setminus\tilde\Gamma$ is the disjoint union of some open $2$-deisks, 
which gives a cell decomposition of $X$. 

We consider an immersion $\varphi:X_{\tilde\Gamma}\to S^3$ 
such that 
\begin{itemize}
 \item 
$\varphi|_{\tilde\Gamma}$ is an embedding,
 \item 
% the image of $\varphi$ is contained in $\Nbd(\varphi(S(X_\Gamma)\cup\Gamma\cup\partial_1 X_\Gamma);S^3)$, and 
$\varphi(X_{\tilde\Gamma})\subset \Nbd(\varphi(\tilde\Gamma);S^3)$, and
 \item 
$\varphi$ is an embedding except on the neighborhood of some triple lines. 
As shown in Figure~\ref{fig:singularity}, 
the image of non-injective points of $\varphi$ form intervals, 
and a neighborhood of each of them is homeomorphic to the union of 
% $\{(x,y,t)\mid x=0,0\leq y\leq1, |t|\leq\pi\}$, 
% $\{(x,y,t)\mid x=-2\sin t/3, y=-2\cos t/3, |t|\leq\pi\}$ and 
% $\{(x,y,t)\mid x=  \sin t/3, y=- \cos t/3, |t|\leq\pi\}$
\begin{align*}
&\{(z,t)\in\C\times\R\mid -1\leq z\leq0, -1\leq t\leq1\}, \\
&\{(z,t)\in\C\times\R\mid z=2re^{\frac{\pi t}{3}\sqrt{-1}}, -1\leq t\leq1,0\leq r\leq1\}, 
\text{ and} \\
&\{(z,t)\in\C\times\R\mid z=re^{-\frac{\pi t}{3}\sqrt{-1}}, -1\leq t\leq1,0\leq r\leq1\}, 
\end{align*}
where we identify $\C\times \R=\R^3$.
\end{itemize}
% Some local pictures of the image of $\varphi$ are shown in the left parts of Figure~\ref{fig:Kirby_diag}-(i) to -(vi). 
% - - - - - - - - - - - - - - - - -
\begin{figure}[tbp]
\centering
\includegraphics[width=.4\hsize]{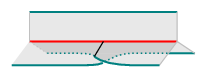}
\caption{The non-injective part of $\varphi$. }
\label{fig:singularity}
\end{figure}
% - - - - - - - - - - - - - - - - -
% - - - - - - - - - - - - - - - - -
% \begin{figure}[tbp]
% \labellist
% \footnotesize\hair 2pt
% \pinlabel {$\varphi(S(X_\Gamma)\setminus T)$} [Br] at 372.76 191.57
% \pinlabel {$\varphi(\Gamma\setminus T)$} [Br] at 367.09 107.94
% \pinlabel {$\varphi(\partial X_\Gamma )$} [Bl] at 403.94  66.84
% \pinlabel {$\varphi(\partial X_\Gamma \setminus \partial X)$} [lB] at 396.85  18.65
% \pinlabel $\tau$ [Bl] at 600.94 212.83
% \pinlabel $\tau$ [Bl] at 600.94 129.20
% \pinlabel (i) [Br] at  26.93 236.92
% \pinlabel (ii) [Br] at  26.93 151.88
% \pinlabel (iii) [Br] at  26.93  66.84
% \pinlabel (vi) [Br] at 372.76  66.84
% \pinlabel (v) [Br] at 372.76 151.88
% \pinlabel (iv) [Br] at 372.76 236.92
% \pinlabel {singularities of $\varphi$} [Bl] at  55.28   8.73
% \endlabellist
% \centering
% \includegraphics[width=1\hsize]{Kirby_diag}
% \caption{Each of the left part shows a local pictures of the images of $\varphi$: 
% (i) a neighborhood of a true vertex, 
% (ii) a neighborhood of a triple line embedded by $\varphi$, 
% (iii) a neighborhood of a triple line immersed by $\varphi$, 
% (iv) a neighborhood of a triple line not intersecting $T$, 
% (v) a neighborhood of a part of $\Gamma$ not intersecting $T$, and 
% (vi) a neighborhood of a part of $\partial_1 X_\Gamma$ not intersecting $T$. 
% Each of the left part shows the corresponding part of the Kirby diagram. }
% \label{fig:Kirby_diag}
% \end{figure}
% - - - - - - - - - - - - - - - - -
% - - - - - - - - - - - - - - - - -
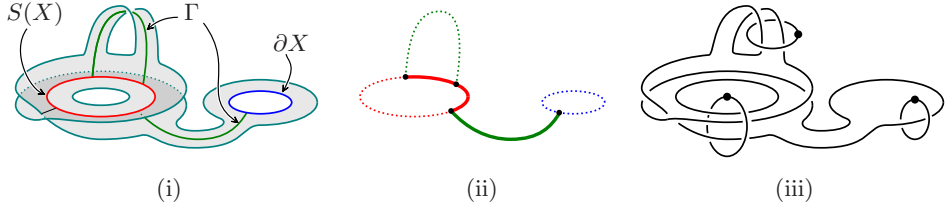
\begin{figure}[tbp]
\labellist
\footnotesize\hair 2pt
\pinlabel (i) [B] at  82.20 8.73
\pinlabel (ii) [B] at  238.11 8.73
\pinlabel (iii) [B] at  394.02 8.73
\pinlabel $\Gamma$ [B] at 93.54 96.60
\pinlabel $\partial X$ [B] at 143.15 81.01
\pinlabel $S(X)$ [B] at 15.34 96.77
\endlabellist
\centering
\includegraphics[width=1\hsize]{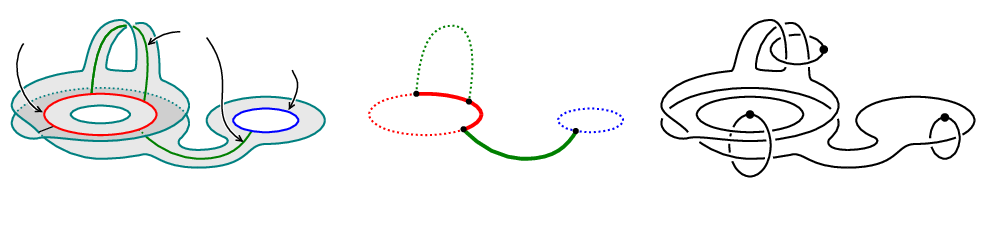}
\caption{An example of how to draw a Kirby diagram. 
(i) The image of $X_{\tilde\Gamma}$ by $\varphi$. 
(ii) The tree graph $T$. 
(iii) The Kirby diagram of a $4$-dimensional thickening of $X$. }
\label{fig:example}
\end{figure}
% - - - - - - - - - - - - - - - - -
See Figure~\ref{fig:example}-(i) for an example of the image of $X_{\tilde\Gamma}$ by $\varphi$. 
The simple polyhedron $X$ we use in this example is encoded by the graph 
\begin{minipage}[c]{17mm}
\includegraphics[width=1\hsize]{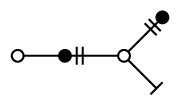}
\end{minipage}.
% (see Subsection~\ref{subsec:Encoding_graph} for the definition of encoding graphs). 

We then encircle each arc in $\varphi(\tilde\Gamma\setminus T)$ 
by a dotted circle so that it does not touch $\varphi(X_\Gamma)$. 
See Figure~\ref{fig:example} for an example. 
If $T$ is chosen as shown in Figure~\ref{fig:example}-(ii), 
we provide dotted circles as shown in Figure~\ref{fig:example}-(iii). 

Let $L_1$ denote the link consisting of those dotted circles, 
and set $L_2=\varphi(\partial X_{\tilde\Gamma}\setminus\partial X)$. 
A Kirby diagram of $W$ that we require 
consists of the dotted circles $L_1$ and the link $L_2$ equipped with some framings. 
See Figure~\ref{fig:example}-(iii). 
% where the framed link $L_2$ is the attaching link of $2$-handles. 
Note that the framings of $L_2$ are determined by the gleam of $X$ and $\varphi$, 
but here we omit the details of those calculation. 
% Although this Kirby diagram is detemied not only by $X$ and $\Gamma$ but also by $i$, 
% we simply call the diagram a \textit{Kirby diagram obtained from} $X$. 
% \begin{remark}
% The second condition in the definition of $\varphi$ is not actually necessary. 
% This is just for convenience for the proof of Lemma~\ref{lem:Heegaard_surf}. 
% \end{remark}
\begin{remark}
In the case $S(X)= \emptyset$, 
$X$ is a (possibly non-orientable) compact surface. 
Then a $4$-dimensional thickening of $X$ is a disk bundle over $X$, 
whose Kirby diagram is easily drawn (see \cite{GS99} for instance). 
\end{remark}
% - - - - - - - - - - - - - - - - - - - - - - - - - - - - - - - - - - - - - - - - -
\subsection{Kirby diagrams to trisections}
\label{subsec:Kirby_diagrams_to_trisections}
% Let $X$ be a shadow of a $4$-manifold $W$ and $\Gamma$ a cut system for $X$. 
% Let $(L_1,L_2)$ be a Kirby diagram obtained as done in the previous subsection. 
%
% Recall that an \textit{unknotting tunnel} of a link $L$ in a $3$-manifold $N$ 
% is a collection $\tau$ of mutually disjoint arcs embedded in $N$ such that 
% $\tau\cap L$ coincides with the disjoint union of the endpoints of $\tau$ 
% and that $\partial\Nbd(L\cup T;N)$ is a Heegaard surface of $N$. 
%
For convenience of constructing a trisection, 
we start with modifying the immersion $\varphi:X_{\tilde\Gamma}\to S^3$. 

% We regard $\tilde\Gamma$ as a graph naturally, 
Let $v_1,\ldots,v_n$ be the vertices of $\tilde\Gamma$ as a graph. 
They are also vertices of the tree graph $T$, 
and let $e_1,\ldots,e_{n-1}$ be the edges of $T$. 
Let $n'$ be the number of edges of $\tilde\Gamma\setminus T$, 
which coincides with $c_1(X)+1$. %as we will see in the proof of Lemma~\ref{lem:genus_of_Sigma}. 
Let $e^\ast_1,\ldots,e^\ast_{n'}$ denote the edges of $\tilde\Gamma\setminus T$. 
We regard $X_{\tilde\Gamma}$ as being decomposed into 
\[
V_1, \ldots,V_n,\quad E_1,\ldots, E_n,\quad E^\ast_1, \ldots , E^\ast_{n'}, 
\]
where 
\begin{align*}
V_i&=\Nbd(v_i;X),\\
% E_j&=\Nbd(e_i;X)\setminus\bigcup_{k=1}^n\Int V_k, \\
E_j&=\Nbd(e_j;X)\setminus\Int (V_1\cup\cdots\cup V_n), \\
% E_\bullet&=\Nbd(\tilde\Gamma\setminus T;X)\setminus\bigcup_{k=1}^n\Int V_k 
E^\ast_k&=\Nbd(e^\ast_k;X)\setminus\Int (V_1\cup\cdots\cup V_n)
\end{align*}
for $i\in\{1,\ldots,n\}$, $j\in\{1,\ldots,n-1\}$ and $k\in\{1,\ldots,n'\}$. 
% \[
% V_i=\Nbd(v_i;X),\quad
% E_i=\Nbd(e_i;X)\setminus\bigcup_{j=1}^n\Int V_j,\quad
% E_\bullet=\Nbd(\tilde\Gamma\setminus T;X)\setminus\bigcup_{j=1}^n\Int V_j. 
% \]
Note that $v_i$ is either a true vertex of $X$ or an endpoint of $\Gamma$. 
If $v_i$ is a true vertex, $V_i$ is as shown in Figure~\ref{fig:VE}-(i). 
If $v_i$ is an endpoint of $\Gamma$ and is on a triple line of $X$, 
$V_i$ is as shown in Figure~\ref{fig:VE}-(ii). 
If $v_i$ is an endpoint of $\Gamma$ and is on $\partial X$, 
$V_i$ is as shown in Figure~\ref{fig:VE}-(iii). 
The portion $E_i$ is shown in Figure~\ref{fig:VE}-(iv) if $e_i\subset S(X)$, 
and it is shown in Figure~\ref{fig:VE}-(v) if $e_i\subset \Gamma$. 
An edge $e^\ast_k$ is contained in either $S(X)$, $\Gamma$ or $\partial X$, 
and hence $E^\ast_k$ is as shown in Figure~\ref{fig:VE}-(iv), (v) or -(vi). 
% - - - - - - - - - - - - - - - - -
\begin{figure}[tbp]
\labellist
\footnotesize\hair 2pt
\pinlabel (i)   [B] at  56.69 82.43
\pinlabel (ii)  [B] at 184.25 82.43
\pinlabel (iii) [B] at 311.81 82.43
\pinlabel (iv)  [B] at  56.69 8.73
\pinlabel (v)   [B] at 184.25 8.73
\pinlabel (vi)  [B] at 311.81 8.73
\endlabellist
\centering
\includegraphics[width=.7\hsize]{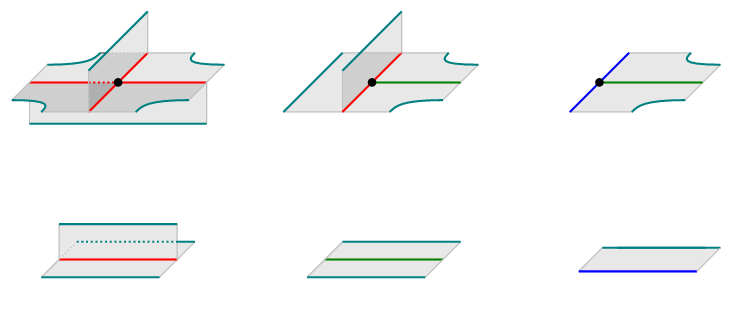}
\caption{The portions $V_i$ and $E_j$. }
\label{fig:VE}
\end{figure}
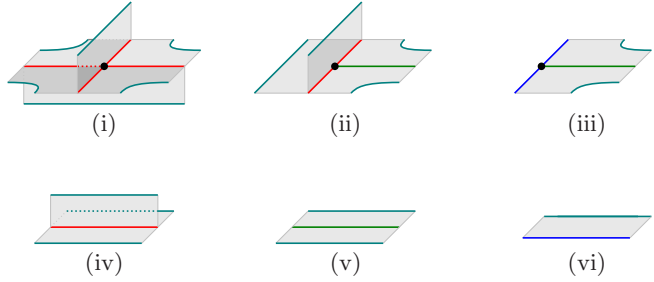
% - - - - - - - - - - - - - - - - -

We can embed each $V_i$ and $E_j$ in a $3$-ball properly, 
and we consider the orientations of these $3$-balls not to be fixed. 
By taking the boundary connected sums of them, 
we can construct the union $V_1\cup\cdots\cup V_n\cup E_1\cup\cdots\cup E_{n-1}$ 
with embedded in a $3$-ball properly since $T$ is a tree. 
The obtained $3$-ball will be denoted by $B_0$, 
in which $\Nbd(T;X)$ is embedded. 
Let us embed $B_0$ into $S^3$. 
We then attach $E^\ast_1,\ldots,E^\ast_{n'}$ to 
$V_1\cup\cdots\cup V_n\cup E_1\cup\cdots\cup E_{n-1}$ 
outside $B_0$ so that their neighborhood are trivial $1$-handles. 
Note that $E^\ast_1,\ldots,E^\ast_{n'}$ may have 
self-intersections as described in Figure~\ref{fig:singularity}.
It determines the immersion $\varphi:X_{\tilde\Gamma}\to S^3$, 
and we define $L_1$ and $L_2$ as done in the previous subsection. 
The Kirby diagram $L_1\sqcup L_2$ near a dotted circle is as shown in Figure~\ref{fig:tau}-(i), -(i)', -(ii) or -(iii), 
where each (i) and (i)' corresponds to a subarc in $S(X)$, 
(ii) corresponds to an arc in $\Gamma$ and 
(iii) corresponds to a boundary component of $X$. 
% - - - - - - - - - - - - - - - - -
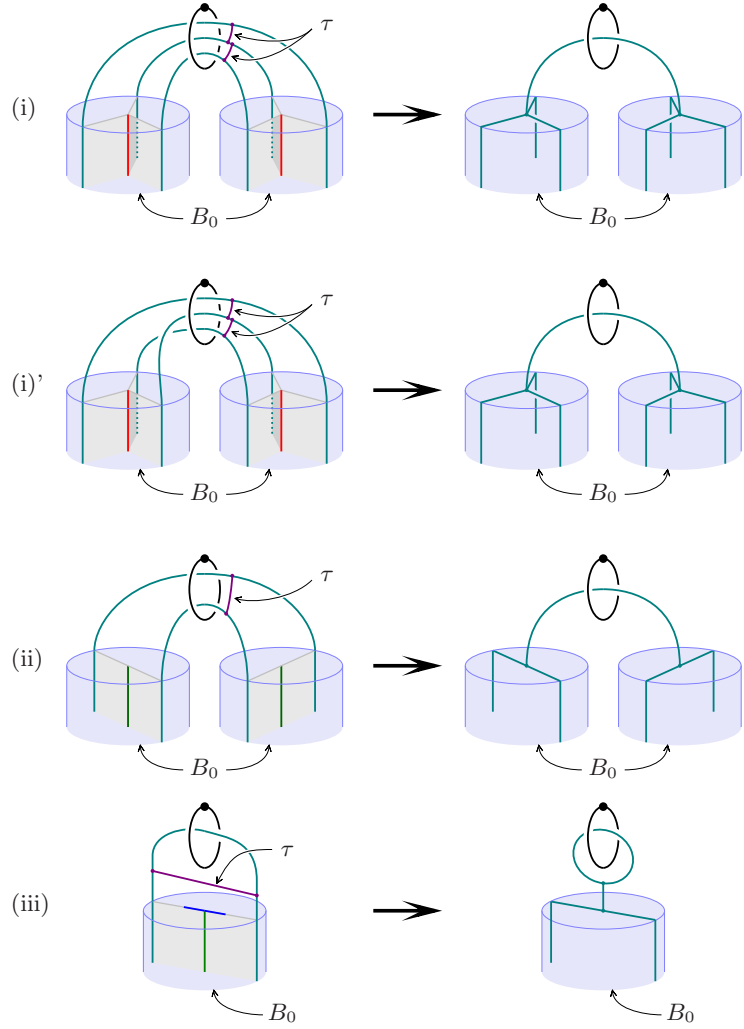
\begin{figure}[tbp]
\labellist
\footnotesize\hair 2pt
\pinlabel {(i)} [Bl] at 5.67 429.68
\pinlabel {(i)'} [Bl] at 5.67 302.12
\pinlabel {(ii)} [Bl] at 5.67 174.56
\pinlabel {(iii)} [Bl] at 5.67 61.17
\pinlabel {$\tau$} [Bl] at 148.82 469.36
\pinlabel {$\tau$} [Bl] at 148.82 214.24
\pinlabel {$\tau$} [Bl] at 148.82 341.80
\pinlabel {$\tau$} [Bl] at 130.39  86.68
\pinlabel {$B_0$} [B] at 97.80 251.68 
\pinlabel {$B_0$} [B] at 97.80 379.23 
\pinlabel {$B_0$} [B] at 97.80 124.12 
\pinlabel {$B_0$} [B] at 282.05 251.68
\pinlabel {$B_0$} [B] at 133.81  10.73
\pinlabel {$B_0$} [B] at 318.06  10.73
\pinlabel {$B_0$} [B] at 282.05 379.23
\pinlabel {$B_0$} [B] at 282.05 124.12
\endlabellist
\centering
\includegraphics[width=.8\hsize]{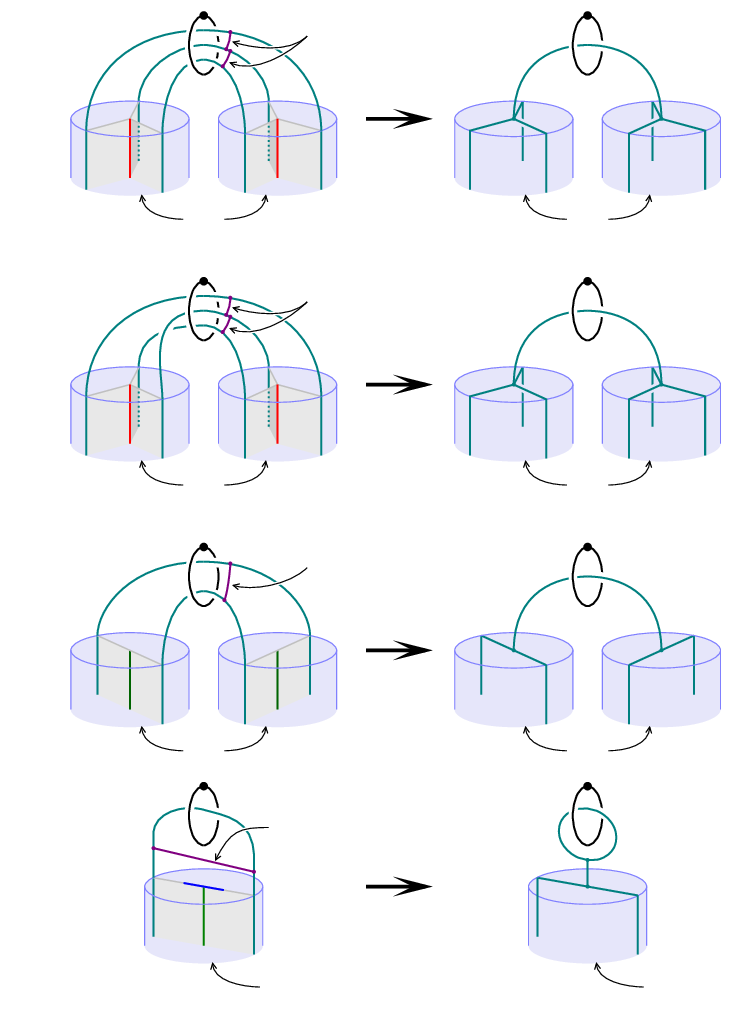}
\caption{The Kirby diagram outside $B_0$. }
\label{fig:tau}
\end{figure}
% - - - - - - - - - - - - - - - - -

Let $N$ be the $3$-manifold obtained from $S^3$ by $0$-surgery along $L_1$, 
and we now regard $L_2$ as a link in $N$. 
Note that $N$ is homeomorphic to $h(S^1\times S^2)$, 
where $h$ is equal to the $1$-weighted complexity $c_1(X)$. 
We then attach some certain arcs to $L_2$ in $N\setminus B_0$ 
near $E^\ast_1,\ldots,E^\ast_{n'}$ 
as shown in the left parts of 
Figures~\ref{fig:tau}-(i), -(i)', -(ii) and -(iii). 
Let $\tau$ denote the collection of those arcs. 
Two arcs are attached for each $e^\ast_k$ in $S(X)$, 
and one arc is attached for each $e^\ast_k$ in $\Gamma$ and in $\partial X$. 
% Note that the number of arcs in $\tau$ is equal to $2+2c_{1/2}(X)$. 
Set $\Sigma=\partial\Nbd(L_2\cup\tau;N)$. 
% The main aim in this subsection is to show the following.  

\begin{lemma}
\label{lem:genus_of_Sigma}
The genus of the surface $\Sigma$ is $3+2c_{1/2}(X)$. 
\end{lemma}
\begin{proof}
One can see that 
\begin{align*}
\chi(\tilde\Gamma)
&=\chi(S(X_\Gamma))+\chi(\Gamma)+\chi(\partial X)-2\chi(\partial \Gamma)\\
&=-c(X_\Gamma)-\chi(\Gamma)\\
&=-c(X)-\sum_{R:\text{region}}(1-\chi(R)). 
\end{align*}
Therefore, 
the number $n'$ of the connected components of $\tilde\Gamma\setminus T$ is equal to $c(X)+\sum_{R:\text{region}}(1-\chi(R))+1$. 
Two arcs in $\tau$ are attached for each edge $e^\ast_k$ in $S(X)$, 
and the number of such edges is $c(X)+1$. 
Hence, the number of arcs in $\tau$ is 
\[
\left(c(X)+\sum_{R:\text{region}}(1-\chi(R))+1\right)
+
\big(c(X)+1\big)
=2+2c_{1/2}(X), 
\]
and the genus of $\Sigma$ is equal to this number plus $1$, 
namely $3+2c_{1/2}(X)$. 
\end{proof}
\begin{lemma}
\label{lem:Heegaard_surf}
The surface $\Sigma$ is a Heegaard surface of $N$. 
\end{lemma}

% We will use the following obvious lemma in the proof of Lemma~\ref{lem:tunnel_number}. 
% \begin{lemma}
% Let $N$ be a closed $3$-manifold and $\Sigma$ an embedded surface in $N$. 
% Suppose that there exist $3$-manifolds $N_1$ and $N_2$ 
% having Heegaard surfaces $\Sigma_1$ and $\Sigma_2$, respectively, 
% such that $(N,\Sigma)\cong(N_1,\Sigma_1)\#(N_2,\Sigma_2)$. 
% Then $\Sigma$ is a Heegaard surface of $N$. 
% \end{lemma}
% \begin{proof}
% Obvious. 
% \end{proof}
% Then we prove Lemma~\ref{lem:tunnel_number}. 

% Figure~\ref{fig:standard_position} shows an embedded graph in $k(S^1\times S^2)$, 
% whose neighborhood gives a Heegaard splitting of $k(S^1\times S^2)$. 
\begin{proof}%[Proof of Lemma~\ref{lem:tunnel_number}]
% Our aim is to show that, by ambient isotopies of $k(S^1\times S^2)$ and simple homotopies of $L_2\cup\tau$, the spacial graph $L_1\sqcup(L_2\cup\tau)$ can be transformed into 
% $\tilde\Gamma$ with some rackets attached. 

% The process is divied into the following four steps. 
% % - - - - - - - - - - - - - - - - -
% \begin{figure}[tbp]
% \labellist
% \small\hair 2pt
% \pinlabel 0 [B] at 26.68 58.11
% \pinlabel 0 [B] at 80.54 58.11
% \endlabellist
% \centering
% \includegraphics[width=.4\hsize]{standard_position}
% \caption{The graph link $L_1\cup L_2\cup\tau$ after ambient isotopies of $S^3$ and simple homotopies of $L_2\cup\tau$. }
% \label{fig:standard_position}
% \end{figure}
% % - - - - - - - - - - - - - - - - -
Outside the $3$-ball $B_0$, 
the spacial graph $L_1\sqcup(L_2\cup\tau)$ can be homotoped as shown in the right parts of 
Figures~\ref{fig:tau}-(i), -(i)', -(ii) and -(iii). 
By our construction, $L_1\sqcup(L_2\cup\tau)$ does not lie in $\Int B_0$, 
and hence $\Nbd(L_2\cup\tau;S^3)$ is a trivial handlebody-knot in $S^3$. 
The dotted circles $L_1$ are meridians of $L_2\cup\tau$. 
It implies that $\Sigma$ is a Heegaard surface of $N$. 
\end{proof}
By Lemmas~\ref{lem:genus_of_Sigma}, \ref{lem:Heegaard_surf} and the construction in Subsection~\ref{subsec:hdl_dec_to_tris}, we have the following. 
\begin{proposition}
For any closed $4$-manifold $W$, 
$g(W)\leq 3 + 2\shco_{1/2}(W)$. 
\end{proposition}
This result will be strengthened in the next subsection. 
\subsection{Proof of Theorem~\ref{thm:complexity_genus}}
We first prove two lemmas regarding conditions for simple polyhedra to be shadows of closed $4$-manifolds. 
\begin{lemma}
\label{lem:closed_surf}
A closed surface $X$ of non-zero genus 
is not a shadow of any closed $4$-manifold. 
\end{lemma}
\begin{proof}
The boundary of a $4$-dimensional thickening of a closed surface $X$ 
is an $S^1$-bundle over $X$. 
Such a $3$-manifold is not homeomorphic to $k(S^1\times S^2)$ 
for any $k\in\Z$ unless the base space $X$ is the $2$-sphere. 
\end{proof}
\begin{lemma}
\label{lem:closed_polyh}
A closed simple polyhedron $X$ having a single region with $S(X)\ne \emptyset$ 
is not a shadow of any closed $4$-manifold. 
\end{lemma}
\begin{proof}
Let $R$ be the unique region of $X$, 
and set $\bar R=R\setminus\Int\Nbd(S(X);X)$. 
Suppose that $\bar R$ is homeomorphic to $\Sigma_{g,b}$. 
Note that $b>0$ by $S(X)\ne \emptyset$. 
Let $M$ be any $4$-dimensional thickening of $X$ 
and $\pi:M\to X$ the projection. 
We suppose that $\partial M$ is homeomorphic to $k(S^1\times S^2)$ for some $k\in\Z_{\geq0}$ 
to lead a contradiction. 

% Suppose $g\ne 0$. 
% Then $\pi^{-1}(\bar R)$ is homeomorphic to $\bar R\times S^1$. 
% Since $\pi^{-1}(\bar R)$ is a codimension $0$ submanifold of $k(S^1\times S^2)$ 
% with $\partial(\pi^{-1}(\bar R))$ consisting of tori, 
% there must exist slopes on $\partial(\pi^{-1}(\bar R))$ such that the Dehn filling 
% along the slopes yields $k'(S^1\times S^2)$ for some $k'\in\Z_{\geq0}$ \cite[Lemma 7.5]{KMN18}. 
% However, such slopes do not exist by the classification of Seifert fibered spaces. 

If $g=0$ and $b=1$, 
$\partial M$ is not homeomorphic to $k(S^1\times S^2)$ for any $k\in\Z_{\geq0}$ 
by \cite[Collorary 3.17]{Cos06b}. 

Suppose that $g\ne0$ or that $g=0$ and $b\geq2$. 
Let $S_1,\ldots,S_m$ be the connected components of $S(X)$, 
and set $N_i=\pi^{-1}(\Nbd(S_i;X))$ for $i\in\{1,\ldots,m\}$. 
Set $N_0=\partial M\setminus\Int(N_1\cup\cdots\cup N_m)$, 
which is homeomorphic to $\bar R\times S^1$. 
Therefore, the $3$-manifold $\partial M$ is decomposed into $N_0,N_1,\ldots,N_m$ along certain embedded tori. 
For each $i\in\{1,\ldots,m\}$, 
$N_i$ is homeomorphic to $\Sigma_{0,3}\times S^1$, $(\Sigma_{0,2},(2,1))$, 
or $(\Sigma_{0,1},(3,1),(3,-1))$ if $S_i$ contains no true vertices, 
and otherwise $N_i$ admits a complete hyperbolic structure with finite volume \cite{CT08}. 
Since all $N_0,N_1,\ldots,N_n$ are irreducible $3$-manifolds, 
$\partial M$ is also irreducible. 
Hence $k=0$, that is, $\partial M$ is $S^3$. 
On the other hand, 
the decomposition $N_0\cup N_1\cup\cdots\cup N_m$ is the canonical one by the irreducibility, 
which contradicts the topology of $S^3$. 
\end{proof}
We next prove Proposition~\ref{prop:thm_case1}, \ref{prop:thm_case2} and \ref{prop:thm_case3}, 
which allows us to show Theorem~\ref{thm:complexity_genus}. 
\begin{proposition}
\label{prop:thm_case1}
Let $X$ be a shadow of a closed $4$-manifold $W$. 
If $X$ is the $2$-sphere or is a surface with boundary, then $g(W)\leq 2 + 2c_{1/2}(X)$. 
\end{proposition}
\begin{proof}
If $X$ is the $2$-sphere, 
then $W$ is diffeomorphic to $S^4$, $\CP$ or $\mCP$. 
In either case, $g(W)\leq 1$, and the lemma holds. 

If $X$ is a surface with boundary, 
$W$ is diffeomorphic to $k(S^1\times S^3)$, 
where $k=2c_{1/2}(X)$. 
It is easy to see that $g(k(S^1\times S^3))=k$, 
and hence the lemma holds. 
\end{proof}
We need the following lemma for the proof of Proposition~\ref{prop:thm_case2}
\begin{lemma}
\label{lem:cut_system}
Let $\alpha=\alpha_1\sqcup\cdots\sqcup\alpha_g$ be a cut system of a $3$-dimensional handlebody 
$H\cong g(S^1\times B^2)$ 
and $\alpha_0, \alpha'_0$ simple closed curves in $\partial H\setminus \alpha$. 
Suppose there exist $i\in\{1,\ldots,g\}$ and orientations of $\alpha_0, \alpha'_0$ and $\alpha_i$ 
such that $[\alpha'_0]-[\alpha_0]=[\alpha_i]$ in $H_1(\partial H)$. 
% $\alpha$ with $\alpha_i$ replaced with either $\alpha_0$ or $\alpha'_0$ 
Let $\tilde\alpha$ and $\tilde\alpha'$ be the collections of curves obtained 
from $\alpha$ by replacing $\alpha_i$ with $\alpha_0$ and $\alpha'_0$, respectively. 
Then either one of $\tilde\alpha$ and $\tilde\alpha'$ is a cut system of $H$. 
Moreover, if there exists a simple closed curve $\gamma\subset \partial H$ such that 
$\gamma$ intersects each $\alpha_0$ and $\alpha_i$ transversely once and $\gamma\cap\alpha_j=\emptyset$ for any $j\in\{1,\ldots,g\}\setminus\{i\}$, 
then $\tilde\alpha$ is a cut system of $H$. 
\end{lemma}
\begin{proof}
We can assume that $i=1$ without loss of generality. 
% If both $\alpha_0$ and $\alpha_0'$ are separating essential curves, the curve obtained by handlesliding $\alpha_0$ over $\alpha_0'$ is also separating. This contradicts $[\alpha'_0]-[\alpha_0]=[\alpha_1]$ in $H_1(\partial H)$ since $\alpha_1$ is a non-separating curve.
% Hence at least one of them is a non-separating curve in $\partial H$. 
Let $D_1\ldots,D_g$ be mutually disjoint disks embedded in $H$ properly such that $\partial D_j=\alpha_j$ for $j\in\{1,\ldots,g\}$. 
Set $V=H\setminus \bigcup_{j=2}^g \Int\Nbd(D_j;H)$. 
It is homeomorphic to a solid torus, 
and $\alpha_0,\alpha'_0$ and $\alpha_1$ are mutually disjoint simple closed curves in $\partial V$. 
Since $[\alpha'_0]-[\alpha_0]=[\alpha_1]$ in $H_1(\partial H)$, 
either one of $\alpha'_0$ or $\alpha_0$ is isotopic to $\alpha_1$ in $\partial V$. 
Assume $\alpha_0$ is isotopic to $\alpha_1$. 
Then, there exists a properly embedded disk $D_0$ in $H$ such that 
\begin{itemize}
 \item
$\partial D_0=\alpha_0$
\item
it does not intersect all $D_1\ldots,D_g$, and  
\item
$D_0$ is isotopic to $D_1$ in $V$. 
\end{itemize}
It follows that $\alpha_0\sqcup\alpha_2\sqcup\cdots\sqcup\alpha_g$ is also a cut system of $H$. 

Then we suppose that a simple closed curve $\gamma$ as in the statement of the lemma exists. 
Since $\gamma$ does not intersect $\alpha_j$ for $j\in\{2,\ldots,g\}$, 
it is also a simple closed curve in $\partial V$, especially a longitude of $V$. 
By the assumption that $\alpha_0$ intersect $\gamma$ transversely once and does not intersect $\alpha_1$, 
the curves $\alpha_0$ and $\alpha_1$ are parallel in $\partial V$. 
Thus, the lemma is proved. 
\end{proof}
\begin{proposition}
\label{prop:thm_case2}
Let $X$ be a shadow of a closed $4$-manifold $W$. 
If $S(X)\ne \emptyset$ and $\partial X=\emptyset$, then $g(W)\leq 2 + 2c_{1/2}(X)$. 
\end{proposition}
\begin{proof}
By Lemma~\ref{lem:closed_polyh}, $X$ has at least two regions, 
and hence $X$ has a triple line $\ell_0$ 
such that at least one of three regions adjacent to $\ell_0$ differs from the others. 
Then we choose a spanning tree $T$ of $\tilde\Gamma$ and 
an immersion $\varphi:X_{\tilde\Gamma}\to S^3$ as considered in 
Subsection~\ref{subsec:Kirby_diagrams_to_trisections}, 
and we can assume that $\ell_0\setminus T\ne \emptyset$ since $S(X)$ is quartic. 
Then we draw a Kirby diagram $L_1\sqcup L_2$ of $W$ as done in Subsection~\ref{subsec:Kirby_diagrams_to_trisections}. 
Note that, for such a Kirby diagram, 
we already have constructed a trisection of $W$ of genus $3 + 2c_{1/2}(X)$, 
so it suffices to show that the genus of this trisection can always decrease by $1$. 

The part of the Kirby diagram $L_1\sqcup L_2$ corresponding to the arc $\ell_0\setminus T$ 
is shown in the left of Figure~\ref{fig:find_triple}-(i) (cf. Figure~\ref{fig:tau}-(i) and -(i)'), 
where $K_1,K_2$ and $K_3$ are the attaching circles of $2$-handles corresponding to the regions adjacent to $\ell_0$. 
By the construction of a trisection in Subsection~\ref{subsec:hdl_dec_to_tris}, 
we obtain a part of trisection diagram as shown in the right of Figure~\ref{fig:find_triple}-(i), 
where we draw some simple closed curves $\delta_1,\ldots,\delta_{10}$. 
Note that $\delta_8,\delta_9$ and $\delta_{10}$ are only partially depicted in the figure. 
By the assumption of $\ell_0$, we can assume either one of the following; 
\begin{enumerate}
 \item[(i)] 
$K_1,K_2$ and $K_3$ are mutually distinct, or
 \item[(ii)]
$K_1$ differs from $K_2=K_3$. 
\end{enumerate}
% - - - - - - - - - - - - - - - - -
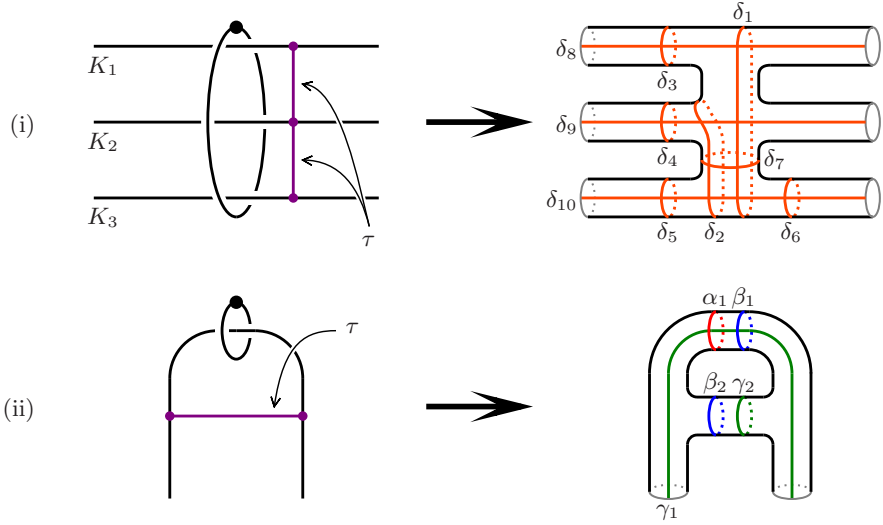
\begin{figure}[tbp]
\labellist
\footnotesize\hair 2pt
\pinlabel {$\delta_1$}   [B] at 237.00 157.98
\pinlabel {$\delta_2$}   [B] at 228.50  92.20
\pinlabel {$\delta_3$}   [B] at 214.33 137.56
\pinlabel {$\delta_4$}   [B] at 214.33 114.88
\pinlabel {$\delta_5$}   [B] at 214.33  92.20
\pinlabel {$\delta_6$}   [B] at 251.18  92.20
\pinlabel {$\delta_7$}   [B] at 246.26 114.05
\pinlabel {$\delta_8$}    [Br] at 188.98 147.06
\pinlabel {$\delta_9$}    [Br] at 188.98 124.38
\pinlabel {$\delta_{10}$} [Br] at 188.98 101.71

\pinlabel {(i)}  [Br] at 26.99 124.38
\pinlabel {(ii)} [Br] at 26.99  39.34

\pinlabel {$K_1$} [Br] at 52.00 142.06
\pinlabel {$K_2$} [Br] at 52.00 119.38
\pinlabel {$K_3$} [Br] at 52.00 96.71

\pinlabel {$\alpha_1$} [B] at 228.50 72.94
\pinlabel {$\beta_1$} [B] at 237.00 72.94
\pinlabel {$\beta_2$} [B] at 228.50 47.43
\pinlabel {$\gamma_1$} [B] at 214.33  9.73
\pinlabel {$\gamma_2$} [B] at 237.00 47.43

\pinlabel {$\tau$} [B]  at 125.04 90.37
\pinlabel {$\tau$} [Bl] at 116.53 64.02

\endlabellist
\centering
\centering
\includegraphics[width=1\hsize]{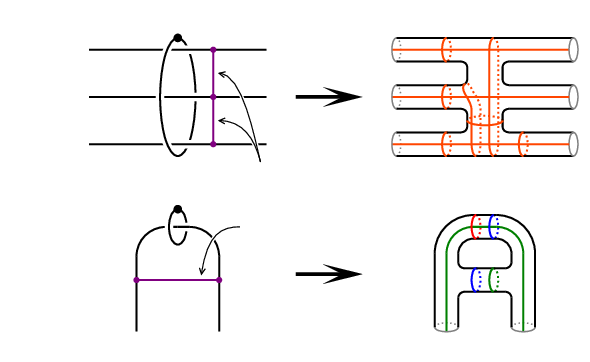}
\caption{Parts of the Kirby diagram $L_1\sqcup L_2$ and the corresponding parts of the central surface.}
\label{fig:find_triple}
\end{figure}
% - - - - - - - - - - - - - - - - -

We first suppose (i). 
As mentioned in Subsection~\ref{subsec:hdl_dec_to_tris}, 
the curves $\alpha=\alpha_1\sqcup\cdots\sqcup\alpha_g$, $\beta=\beta_1\sqcup\cdots\sqcup\beta_g$ and $\gamma=\gamma_1\sqcup\cdots\sqcup\gamma_g$ of 
a trisection diagram $(\Sigma_g,\alpha,\beta,\gamma)$ of $W$ can be chosen so that 
\begin{itemize}
 \item
$\alpha_1=\delta_1$,  
 \item 
$\beta_{1}=\delta_3,\ \beta_{2}=\delta_4$ and $\beta_{3}=\delta_5$, and 
 \item 
$\gamma_{1}=\delta_8,\ \gamma_{2}=\delta_9$ and $\gamma_{3}=\delta_{10}$. 
\end{itemize}
Note that $\gamma_1,\gamma_2$ and $\gamma_3$ come from $K_1,K_2$ and $K_3$, respectively. 
Let $\beta'_{1}$ be $\delta_1$, which is obtained from $\beta_{1}$ by handle sliding over $\beta_{2}$ and then over $\beta_{3}$. 
Then $(\Sigma_g,\alpha,\beta',\gamma)$ is also a trisection diagram of $W$, where $\beta'=\beta'_1\sqcup\beta_2\sqcup\cdots\sqcup\beta_g$. 
Since the triple $(\alpha_1,\beta'_{1},\gamma_{1})$ forms a destabilization triple, 
we obtain $g(W)\leq 2 + 2c_{1/2}(X)$ by a destabilization. 

We next suppose (ii). 
As mentioned in Subsection~\ref{subsec:hdl_dec_to_tris}, 
the curves $\alpha=\alpha_1\sqcup\cdots\sqcup\alpha_g$, $\beta=\beta_1\sqcup\cdots\sqcup\beta_g$ and $\gamma=\gamma_1\sqcup\cdots\sqcup\gamma_g$ of 
a trisection diagram $(\Sigma_g,\alpha,\beta,\gamma)$ of $W$ can be chosen so that 
\begin{itemize}
 \item
$\alpha_1=\delta_1$,  
 \item 
$\beta_{1}=\delta_3,\ \beta_{2}=\delta_4$ and $\beta_{3}=\delta_7$, and 
 \item 
$\gamma_{1}=\delta_8,\ \gamma_{2}=\delta_9(=\delta_{10})$ and $\gamma_{3}=\delta_{7}$. 
\end{itemize}
Note that $\gamma_1$ and $\gamma_2$ come from $K_1$ and $K_2(=K_3)$, respectively. 
% Set $\beta_{4,\ldots,g}=\beta_4\sqcup\cdots\sqcup\beta_g$. 
By Lemma~\ref{lem:cut_system}, 
$\beta_{3}$ can be replaced with another curve $\beta'_3$, 
where $\beta'_3$ is either $\delta_5$ or $\delta_6$. 
Suppose that $\beta'_3=\delta_5$, that is, 
three curves of $\beta$ can be chosen as $\delta_3,\delta_4$ and $\delta_5$. 
Then we can find a destabilization triple in the same way as in (i), 
and we obtain $g(W)\leq 2 + 2c_{1/2}(X)$ by a destabilization. 
Then suppose that $\beta'_3=\delta_6$, 
and set $\displaystyle \beta'=\beta_1\sqcup\beta_2\sqcup\beta_3'\sqcup\big(\beta_{4}\sqcup\cdots\sqcup\beta_g\big)$. 
% and set $\displaystyle \beta'=\beta_1\sqcup\beta_2\sqcup\beta_3'\sqcup\left(\bigsqcup_{j=4}^g\beta_{j}\right)$. 
% \Erase{We handle slide $\beta_{1}$ over $\beta_{2}$ to get another curve $\beta'_{1}$, 
% so that we obtain another trisection diagram $(\Sigma_g,\alpha,\beta',\gamma)$ of $W$, 
% where $\beta'=\beta'_1\sqcup\beta_2\sqcup\cdots\sqcup\beta_{g}$. 
% Note that $\beta'_{1}=\delta_2$. }
We note that $[\delta_1]-[\delta_2]=[\beta_1]$ in $H_1(\Sigma_g)$ for some orientations. 
Since $\delta_8$ is a simple closed curve intersecting $\beta'$ exactly once at a point of $\beta_{1}$, 
we can replace $\beta_{1}$ with $\delta_1$, which will be denoted by $\beta'_{1}$, by Lemma~\ref{lem:cut_system}. 
Hence, $(\Sigma_g,\alpha,\beta'',\gamma)$ is also a trisection diagram of $W$, where $\beta''=\beta'_1\sqcup\beta_2\sqcup\beta_3'\sqcup\big(\beta_{4}\sqcup\cdots\sqcup\beta_g\big)$. 
Then the triple $(\alpha_1, \beta''_{1}, \gamma_{1})$ is a destabilization one, 
and we obtain $g(W)\leq 2 + 2c_{1/2}(X)$ by a destabilization. 
\end{proof}
\begin{proposition}
\label{prop:thm_case3}
Let $X$ be a shadow of a closed $4$-manifold $W$. 
If $S(X)\ne \emptyset$ and $\partial X\ne\emptyset$, then $g(W)\leq 2 + 2c_{1/2}(X)$. 
\end{proposition}
\begin{proof}
Let $T$ be a spanning tree of $\Gamma$ and 
$\varphi$ an immersion $X_{\tilde\Gamma}\to S^3$ as considered in 
Subsection~\ref{subsec:Kirby_diagrams_to_trisections}. 
Then we draw a Kirby diagram $L_1\sqcup L_2$ of $W$ as done in Subsection~\ref{subsec:Kirby_diagrams_to_trisections}. 
For such a Kirby diagram, 
we have already constructed a trisection of $W$ of genus $3 + 2c_{1/2}(X)$. 

Since $\partial X\ne\emptyset$, 
the Kirby diagram $L_1\sqcup L_2$ contains a part as shown in the left of Figure~\ref{fig:find_triple}-(ii) (cf. Figure~\ref{fig:tau}-(iii)). 
By the construction of a trisection in Subsection~\ref{subsec:hdl_dec_to_tris}, 
we obtain a part of trisection diagram as shown in the right of Figure~\ref{fig:find_triple}-(ii). 
Moreover, a trisection diagram $(\Sigma_g,\alpha,\beta,\gamma)$ of $W$ can be drawn 
so that simple closed curves $\alpha_1,\beta_1,\beta_2,\gamma_1$ and $\gamma_2$ of 
$\alpha=\alpha_1\sqcup\cdots\sqcup\alpha_g$, $\beta=\beta_1\sqcup\cdots\sqcup\beta_g$ and $\gamma=\gamma_1\sqcup\cdots\sqcup\gamma_g$ 
are as shown in the right of Figure~\ref{fig:find_triple}-(ii). 
In this diagram, $(\alpha_1,\beta_1,\gamma_1)$ is a destabilization triple, 
and hence we get $g(W)\leq 2 + 2c_{1/2}(X)$. 
\end{proof}

We are now ready to prove Theorem~\ref{thm:complexity_genus}. 
\begin{theorem}
\label{thm:complexity_genus}
For any closed $4$-manifold $W$ and any real number $r\geq1/2$, 
$g(W)\leq 2 + 2\shco_r(W)$. 
\end{theorem}
\begin{proof}
Let $X$ be any shadow of $W$. 
It is enough to show the inequality $g(W)\leq 2 + 2c_{1/2}(X)$ 
since $\shco_{1/2}(W)\leq\shco_{r}(W)$ by Proposition~\ref{prop:rel_sh}. 
By Lemmas~\ref{lem:closed_surf} and \ref{lem:closed_polyh}, 
at least one of the following holds; 
\begin{itemize}
 \item
$X$ is the $2$-sphere or a surface with boundary, 
 \item 
$S(X)\ne \emptyset$ and $\partial X = \emptyset$, or 
 \item 
$S(X)\ne \emptyset$ and $\partial X \ne \emptyset$. 
\end{itemize}
In either case, 
we have $g(W)\leq 2 + 2\shco_{1/2}(X)$ by Propositions~\ref{prop:thm_case1}, \ref{prop:thm_case2} and \ref{prop:thm_case3}. 
\end{proof}
% \begin{remark}
% We have also shown in this section that 
% if a closed $4$-manifold admits a shadow having $b$ boundary components and $r$ regions, 
% then the $4$-manifold admits a $(g;k_1,k_2,k_3)$-trisection, 
% where 
% \begin{align*}
% g   &= 2 + 2c_{1/2}(W),\\ 
% k_1 &= c_1(X),\\ 
% k_2 &= 3 + 2c_{1/2}(W)-r,\\ 
% k_3 &= 1-c_1(X)+r-\chi(W). 
% \end{align*}
% \end{remark}
\subsection{Examples}
\label{subsec:Examples}
In this section, we will determine the exact values of $\shco_{1/2}$ for infinite families of certain $4$-manifolds by using Theorem~\ref{thm:complexity_genus}. 

Now we define a simple polyhedron $X_k$ for $k\in\Z_{\geq1}$. 
Let $X_1$ be the $2$-sphere, which is encoded by a graph shown in Figure~\ref{fig:Xk}-(i). 
For $k\geq2$, 
% $X_k$ is defined as follows. 
let $C_1,\ldots,C_{k-1}$ be simple closed curves in $X_1$ such that they split $X_1$ into two disks and $k-2$ annuli. 
Then $X_k$ is defined as a simple polyhedron obtained from $X_1$ by attaching $2$-disks $D_1,\ldots,D_{k-1}$ 
% so that $\partial D_i$ is identified with $\partial C_i$ for $i\in\{1,\ldots,k-1\}$. 
along their boundaries to $C_1,\ldots,C_{k-1}$, respectively. 
The polyhedron $X_k$ is shown in Figure~\ref{fig:Xk3} and encoded in Figure~\ref{fig:Xk}-(iii). 
Note that $\mathrm{rank} H_2(X_k)=k$ and $c_{1/2}(X_k)=\max\left\{0,\frac{k-2}2\right\}$. 
% - - - - - - - - - - - - - - - - -
\begin{figure}[tbp]
\labellist
\footnotesize\hair 2pt
\pinlabel {(i)}     [Br] at  20.05 26.93
\pinlabel {(ii)}    [Br] at  98.00 26.93
\pinlabel {(iii)}   [Br] at 197.21 26.93
\pinlabel {$k-1$}   [B]  at 255.74 46.77
\endlabellist
\centering
\includegraphics[width=.6\hsize]{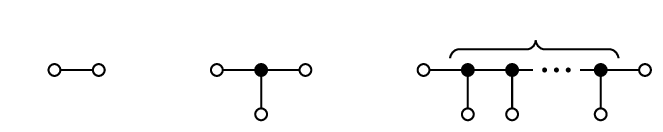}
\caption{Encoding graphs of (i) $X_1$, (ii) $X_2$ and (iii) $X_k$. }
\label{fig:Xk}
% \pinlabel {$k+1$} [B] at 289.76 46.77
% \pinlabel {(i)}   [Br] at  19.05 21.26
% \pinlabel {(ii)}  [Br] at 230.23 21.26
% \pinlabel {$R_0$}   [B] at  33.22 19.84
% \pinlabel {$Q_0$}   [B] at  47.40 19.84
% \pinlabel {$R_1$}   [B] at  75.74 19.84
% \pinlabel {$Q_1$}   [B] at  61.57 19.84
% \pinlabel {$R_2$}   [B] at  89.91 19.84
% \pinlabel {$R_k$}   [B] at 146.61 19.84
% \pinlabel {$Q_k$}   [B] at 132.43 19.84
% \pinlabel {$R_{k+1}$} [B] at 160.78 19.84
% \endlabellist
% \centering
\includegraphics[width=.35\hsize]{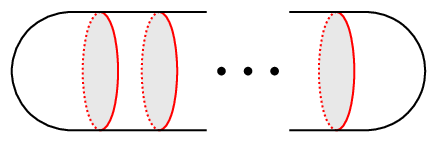}
\caption{The simple polyhedron $X_k$. }
\label{fig:Xk3}
\end{figure}
% - - - - - - - - - - - - - - - - -
\begin{proposition}
\label{prop:Examples}
For any non-negative integers $k_1, k_2$ and $k_3$, 
% \begin{enumerate}
%  \item 
% $\shco_{1/2}(k\CP)=\shco_{1/2}(k\mCP)=\max\left\{0,\frac{k-2}2\right\}$. 
%  \item 
% $\shco_{1/2}(k(S^2\times S^2))=k-1$. 
% \end{enumerate}
\[
\shco_{1/2}\left(k_1(S^2\times S^2)\# k_2\CP\# k_3\mCP\right)=\max\left\{0,\frac{2k_1+k_2+k_3-2}2\right\}. 
\]
\end{proposition}
\begin{proof}
Set $k=2k_1+k_2+k_3$ and $W=k_1(S^2\times S^2)\# k_2\CP\# k_3\mCP$. 
If $k=0$, the equality holds since $W\ (=S^4)$ admits a shadow homeomorphic to the sphere whose $1/2$-weighted complexity is $0$. 

Suppose $k\geq1$. 
The simple polyhedron $X_k$ can be embedded in $W$ as a shadow, 
and hence $\shco_{1/2}(W)\leq c_{1/2}(X_k)=\max\left\{0,\frac{k-2}2\right\}$. 
% Note that gleams of $D_0^-$ and the annular regions are all given by $\pm1$, and those of the other regions are given by $0$. 
On the other hand, since $g(W)=k$, 
we have $\shco_{1/2}(W)\geq \frac{k-2}2$ by Theorem~\ref{thm:complexity_genus}. 
The value of $\shco_{1/2}$ must not be negative. 
We obtain $\shco_{1/2}(W)=\max\left\{0,\frac{k-2}2\right\}$. 
\end{proof}

\begin{remark}
\label{rmk:best_result}
\begin{enumerate}
 \item 
By considering the same shadow $X_k$ of $k\CP$, 
we also have $\shco_{r}(k\CP)\leq \max\left\{0,(k-2)r\right\}$ for $0\leq r<1/2$. 
It follows that $k\CP$ violates the inequality $g\leq2+2\shco_{r}$ for $0\leq r<1/2$ and $k\geq3$, 
and the minimum of $r$ satisfying the inequality in Theorem~\ref{thm:complexity_genus} is $1/2$. 
 \item 
The examples given in Proposition~\ref{prop:Examples} attain all the pairs 
$(g,\shco_{1/2})\in\Z_{\geq0}\times \frac12\Z_{\geq0}$ satisfying the equality $g=2+2\shco_{1/2}$. 
Therefore, the inequality $g\leq2+2\shco_{1/2}$ shown in Theorem~\ref{thm:complexity_genus} is the best possible result. 
\end{enumerate}
\end{remark}
%==================================================================================
\section{Closed $4$-manifolds with $\shco_{1/2}\leq1/2$}
% - - - - - - - - - - - - - - - - - - - - - - - - - - - - - - - - - - - - - - - - -
This section is mainly devoted to the proof of Theorem~\ref{thm:complexity1/2}, 
which, in conjunction with Theorem~\ref{thm:complexity0}, provides the classification of all closed $4$-manifolds with $\shco_{1/2}\leq1/2$. 
We start with exhibit simple polyhedra with $c_{1/2}\leq1/2$. 

\subsection{Simple polyhedra with $c_{1/2}\leq1/2$}
% In this subsection, we exhibit simple polyhedra with $c_{1/2}\leq1/2$. 
% Let $X$ be a simple polyhedron with $S(X)\ne\emptyset$ or homeomorphic to $S^2$. 
Let $X$ be a simple polyhedron such that it is not homeomorphic to a closed surface or is homeomorphic to $S^2$. 

We first consider the case $c_{1/2}(X)=0$. 
Then $X$ is homeomorphic to $S^2$, or it is a special polyhedron without true vertices. 
% Hence a region of $X$ is a $2$-disk or $\RP^2$. 
% Therefore, the simple polyhedra with $c_{1/2}=0$ are shown in Figure~\ref{fig:c=0}. 
% - - - - - - - - - - - - - - - - -
% \begin{figure}[tbp]
% \labellist
% \footnotesize\hair 2pt
% \pinlabel (p1) [Br] at  26.91 19.84
% \pinlabel (d1) [Br] at 154.47 19.84
% \pinlabel (d2) [Br] at 282.03 19.84
% \pinlabel (d3) [Br] at 409.59 19.84
% \pinlabel (d4) [Br] at 537.15 19.84
% \endlabellist
% \centering
% \includegraphics[width=1\hsize]{c=0}
% \caption{Simple polyhedra with $c_{1/2}=0$.}
% \label{fig:c=0}
% \end{figure}
% - - - - - - - - - - - - - - - - -
% The simple polyhedra shown in Figures~\ref{fig:c=0}-(p1), -(d1),$\ldots$, -(d4) 
% will be denoted by $X_{\rm (p1)}, X_{\rm (d1)},\ldots,X_{\rm (d4)}$, respectively. 
The closed $4$-manifolds in which $S^2$ is embedded as shadows are only $S^4$, $\CP$ and $\mCP$. 
The closed $4$-manifolds with $\spshco=0$ are classified by Costantino in \cite{Cos06b}, 
and thus we have the following. 
\begin{theorem}
[{cf. \cite[Theorem 1.1]{Cos06b}}]
\label{thm:complexity0}
The $1/2$-weighted shadow-complexity of a closed $4$-manifold $W$ is $0$ 
if and only if $W$ is diffeomorphic to either one of 
$S^4$, $\CP$, $\mCP$, $S^2\times S^2$, $2\CP$, $\CP\#\mCP$ or $2\mCP$. 
% \[
% S^4,\ \CP,\ \mCP,\ S^2\times S^2,\ \CP\#\CP,\ \CP\#\mCP\text{ or }\mCP\#\mCP. 
% \]
\end{theorem}

We next consider the case $c_{1/2}(X)=1/2$. 
Then $X$ has no true vertices, 
and all regions of $X$ are $2$-disks except one region $R_0$. \
The Euler characteristic $\chi(R_0)$ of $R_0$ is $0$, 
and hence $R_0$ is an annulus or a M\"obius band. 
Therefore, the simple polyhedra with $c_{1/2}=1/2$ are shown in Figure~\ref{fig:c=1/2}.
% - - - - - - - - - - - - - - - - -
\begin{figure}[tbp]
\labellist
\footnotesize\hair 2pt
% \pinlabel (t1) [Br] at  35.57 304.72
% \pinlabel (k1) [Br] at  198.57 304.72
\pinlabel (a1) [Br] at  35.57 248.03
\pinlabel (a2) [Br] at 191.48 248.03
\pinlabel (a3) [Br] at 347.39 248.03
\pinlabel (a4) [Br] at 503.29 248.03

\pinlabel (a5) [Br] at  35.57 191.34
\pinlabel (a6) [Br] at 191.48 191.34
\pinlabel (a7) [Br] at 347.39 191.34
\pinlabel (a8) [Br] at 503.29 191.34

\pinlabel (a9) [Br] at  35.57 134.65
\pinlabel (a10) [Br] at 191.48 134.65
\pinlabel (a11) [Br] at 347.39 134.65
\pinlabel (a12) [Br] at 503.29 134.65

\pinlabel (a13) [Br] at  35.57 77.95
\pinlabel (a14) [Br] at 163.13 77.95
\pinlabel (a15) [Br] at 290.69 77.95
\pinlabel (a16) [Br] at 418.25 77.95
\pinlabel (a17) [Br] at 545.81 77.95

\pinlabel (m1) [Br] at  35.57 21.26
\pinlabel (m2) [Br] at 163.13 21.26
\pinlabel (m3) [Br] at 290.69 21.26
\pinlabel (m4) [Br] at 418.25 21.26
\pinlabel (m5) [Br] at 545.81 21.26

\endlabellist
\centering
\includegraphics[width=1\hsize]{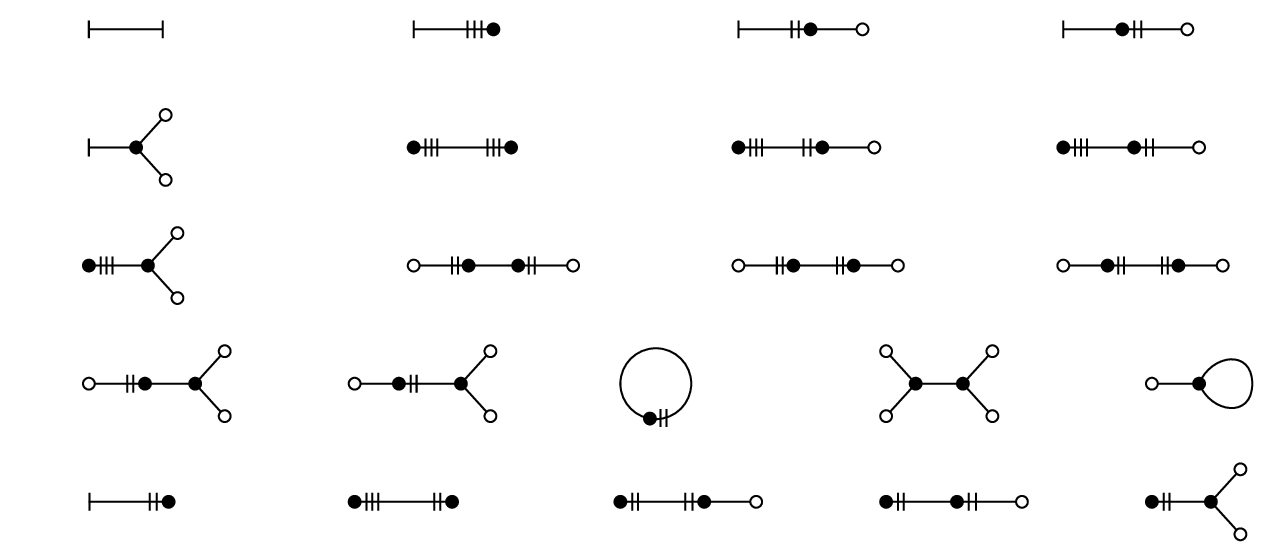}
\caption{Simple polyhedra with $c_{1/2}=1/2$.}
\label{fig:c=1/2}
\end{figure}
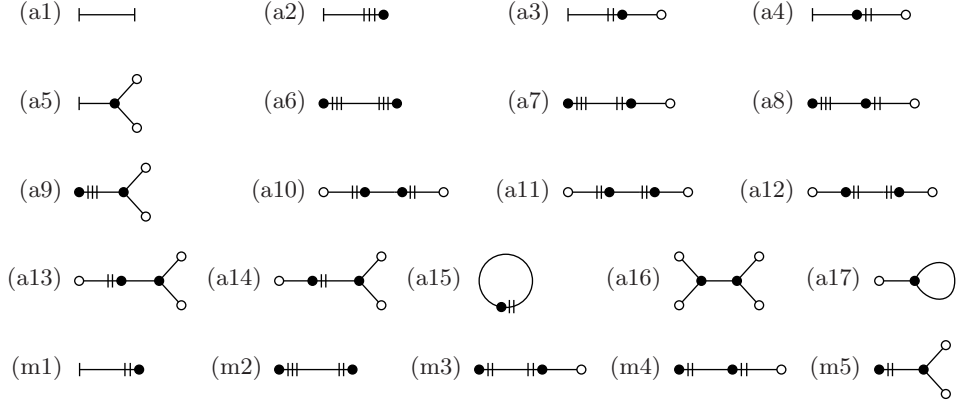
% - - - - - - - - - - - - - - - - -
The simple polyhedra encoded in 
Figures~\ref{fig:c=1/2}-(a1),$\ldots$, -(a14), -(a16), -(m1),$\ldots$, -(m5) 
will be denoted by $X_{\rm (a1)},\ldots,X_{\rm (a14)}, X_{\rm (a16)}, X_{\rm (m1)},\ldots,X_{\rm (m5)}$, respectively. 

Each encoding graph shown in Figures~\ref{fig:c=1/2}-(a15) and -(a17) has a cycle, 
it can not determine a simple polyhedron uniquely. 
Actually, each of them corresponds to exactly two simple polyhedra up to homeomorphisms. 
% The graph (t1) with cohomology class $0$ corresponds to a torus, 
% and that with cohomology class $1$ corresponds to a Klein bottle. 
% A Klein bottle is already represented by $X_{\rm (k1)}$, 
% so let $X_{\rm (t1)}$ denote the simple polyhedron encoded by the graph (t1) 
% with cohomology class $0$, namely a torus. 
% The simple polyhedra encoded by the graph (t1) with cohomology class $0$ and $1$, 
% respectively, will be denoted by $X_{\rm (t1)}^0$ and $X_{\rm (t1)}^1$. 
% Note that $X_{\rm (t1)}^0$ is a torus, 
% and $X_{\rm (t1)}^1$ is a Klein bottle $X_{\rm (k1)}$. 
% The simple polyhedra encoded by the graph (a15) with cohomology classes $0$ and $1$ 
% and that by the graph (a17) with cohomology classes $0$ and $1$ are 
% denoted by $X_{\rm (a15)}^0,X_{\rm (a15)}^1,X_{\rm (a17)}^0$ and $X_{\rm (a17)}^1$, 
% respectively. 
Let $X_{\rm (a15)}^0$ and $X_{\rm (a15)}^1$ be simple polyhedra described in Figures~\ref{fig:a15}-(i) and -(ii), respectively, 
which are encoded by the graph shown in Figures~\ref{fig:c=1/2}-(a15). 
% - - - - - - - - - - - - - - - - -
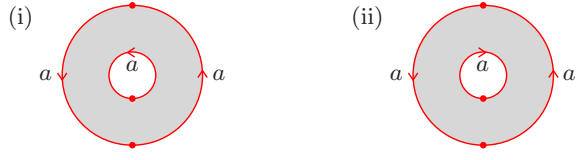
\begin{figure}[tbp]
\labellist
\footnotesize\hair 2pt
\pinlabel $a$ [Bl]  at 121.10 47.36
\pinlabel $a$ [Bl]  at 333.69 47.36
\pinlabel $a$ [B]   at  75.74 53.03
\pinlabel $a$ [B]   at 288.34 53.03
\pinlabel $a$ [Br]  at  30.39 47.36
\pinlabel $a$ [Br]  at 242.99 47.36
\pinlabel (i)  [Br] at  19.05  81.3
\pinlabel (ii) [Br] at 231.65 81.37
\endlabellist
\centering
\includegraphics[width=.6\hsize]{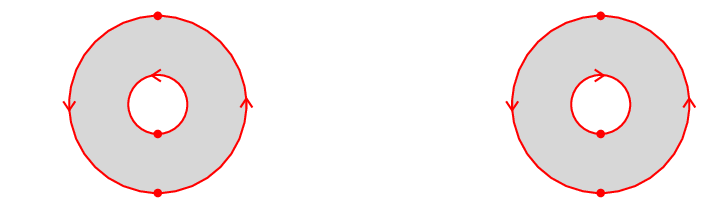}
\caption{Simple polyhedra (i) $X_{\rm (a15)}^0$ and (ii) $X_{\rm (a15)}^1$.}
\label{fig:a15}
\end{figure}
% - - - - - - - - - - - - - - - - -
Let $X_{\rm (a17)}^0$ be a simple polyhedron obtained from a torus by gluing a $2$-disk along its boundary to a meridian of the torus. 
We also define $X_{\rm (a17)}^1$ as a simple polyhedron obtained from Klein bottle by gluing a $2$-disk along its boundary to a simple closed curve representing $x$ in the fundamental group $\langle x,y\mid xyxy^{-1} \rangle$. 
Both $X_{\rm (a17)}^0$ and $X_{\rm (a17)}^1$ are simple polyhedra encoded by the graph shown in Figures~\ref{fig:c=1/2}-(a17). 
% - - - - - - - - - - - - - - - - - - - - - - - - - - - - - - - - - - - - - - - - -
\subsection{Useful facts}
Here we state some useful facts about shadows of closed $4$-manifolds and the elementary ideals of finitely generated free abelian groups. 
\begin{lemma}
[Costantino {\cite[Lemma 3.12]{Cos06b}}]
\label{lem:Costantino}
Let $X$ be a simple polyhedron. 
If $H_2(X)=0$ and $\mathrm{tor}H_1(X)\ne0$, 
then $\partial M_{(X,\gl)}$ is not homeomorphic to $k(S^1\times S^2)$ 
for any gleam $\gl$ and integer $k$, 
especially, $X$ is not a shadow of any closed $4$-manifold. 
\end{lemma}

Martelli classified all the closed $4$-manifolds with $\shco=0$ and finite fundamental group in \cite{Mar11}. 
The following is a partial result of him. 
\begin{theorem}
[Martelli {\cite[Theorem 1.7]{Mar11}}]
\label{thm:finite_pi1}
A closed $4$-manifold $W$ has shadow-complexity $0$ and $|\pi_1(W)|\leq3$ 
if and only if $W$ is diffeomorphic to 
\[
W'\#h(S^2\times S^2)\#k\CP\#l\mCP
\]
for some $h,k,l\in\Z$, where $W'$ is $S^4$, $\mathcal{S}_2$, $\mathcal{S}'_2$ or $\mathcal{S}_3$. 
\end{theorem}
% \begin{lemma}
% \label{lem:S^1-bundle}
% An orientable $S^1$-bundle over a closed surface $\Sigma_{g}$ of genus $g$ 
% with Euler number $e$ is not homeomorphic to $k(S^1\times S^2)$ 
% unless $(g,e,k)=(0,0,1)$ or $(0,1,0)$.
% \end{lemma}
\begin{lemma}
\label{lem:elementary_ideal}
For any non-negative integer $k$, 
the $d$-th elementary ideal of $\pi_1(k(S^1\times S^2))$ is 
isomorphic to $(0)$ if $d<k$, and $(1)=\Z[t^{\pm1}_1,\ldots,t^{\pm1}_k]$ if $k\leq d$. 
\end{lemma}
% \begin{lemma}
% \label{lem:reduceY12}
% Suppose a clsoed $4$-manifold $W$ admits a shadow encoded by a graph shown in Figure~\ref{fig:reduceY12}-(i). 
% Then $W$ also admits a shadow encoded by at least one of graphs shown in Figure~\ref{fig:reduceY12}-(ii) and -(iii). 
% \end{lemma}
% % - - - - - - - - - - - - - - - - -
% \begin{figure}[tbp]
% \labellist
% \footnotesize\hair 2pt
% \pinlabel $G$ [B] at 38.89 18.43
% \pinlabel $G$ [B] at 180.62 18.43
% \pinlabel $G$ [B] at 294.01 18.43
% \pinlabel (i) [Br] at 19.05 21.26
% \pinlabel (ii) [Br] at  160.78 21.26
% \pinlabel (iii) [Br] at 274.17  21.26
% \endlabellist
% \centering
% \includegraphics[width=.7\hsize]{reduceY12}
% \caption{Encoding graph shown in (i) can be replaced that in (ii) or (iii).}
% \label{fig:reduceY12}
% \end{figure}
% % - - - - - - - - - - - - - - - - -
% - - - - - - - - - - - - - - - - - - - - - - - - - - - - - - - - - - - - - - - - -
\subsection{Non-existence}
In the following Lemmas~\ref{lem:a6-8_m3-4}, 
%\ref{lem:p1_t1_k1}, %\ref{lem:a3-5}, 
\ref{lem:a15} and \ref{lem:m2}, 
we will show that the simple polyhedra 
% $X_{\rm (p1)}$, $X_{\rm (t1)}^0$, $X_{\rm (t1)}^1$, $X_{\rm (k1)}$, 
$X_{\rm (a6)}$, 
$X_{\rm (a7)}$, $X_{\rm (a8)}$, $X_{\rm (a15)}^0$, $X_{\rm (a15)}^1$, %$X_{\rm (a17)}^0$, $X_{\rm (a17)}^1$, 
$X_{\rm (m3)}$ and $X_{\rm (m4)}$ are not shadows of closed $4$-manifolds. 
% \begin{lemma}
% \label{lem:p1_t1_k1}
% The simple polyhedra $X_{\rm (p1)}, X_{\rm (t1)}^0, X_{\rm (k1)}(=X_{\rm (t1)}^1)$ are not shadows of closed $4$-manifolds. 
% \end{lemma}
% \begin{proof}
% It follows directly from Lemma~\ref{lem:closed_surf}. 
% \end{proof}
\begin{lemma}
\label{lem:a6-8_m3-4}
The simple polyhedra $X_{\rm (a6)}, X_{\rm (a7)}$, $X_{\rm (a8)}$, $X_{\rm (a15)}^1$, $X_{\rm (m3)}$ and $X_{\rm (m4)}$
are not shadows of closed $4$-manifolds. 
\end{lemma}
\begin{proof}
The second homology groups of 
simple polyhedra $X_{\rm (a6)}, X_{\rm (a7)}$, $X_{\rm (a8)}$, $X_{\rm (a15)}^1$, $X_{\rm (m3)}$ and $X_{\rm (m4)}$ all vanish, 
and their first homology groups are $\Z/3\Z$, $\Z/3\Z$, $\Z/6\Z$, $\Z/3\Z$, $\Z/2\Z$ and $\Z/4\Z$, respectively. 
Hence, the lemma follows from Lemma~\ref{lem:Costantino}. 
\end{proof}
\begin{lemma}
\label{lem:a15}
The simple polyhedron $X_{\rm (a15)}^0$ 
is not a shadow of closed $4$-manifolds. 
\end{lemma}
\begin{proof}
% We first consider $X_{\rm (a15)}^0$. 
% We equip arbitrary gleam with $X_{\rm (a15)}^0$, 
% and let $M$ denote its $4$-dimensional thickening. 
% Figure~\ref{fig:Kirby_a15} shows a Kirby diagram of $M$, where $m$ is some integer. 
% Therefore, the fundamental group of $\partial M$ is 
% \[
% \langle
% x,y,z
% \mid
% xzx^{-1}z^{-1},x^{n-1}y^{-1}zyxy^{-1}zyz^{-1}, yx^{-1}y^{-1}xzyx^{-1}y^{-1}z^{-1}
% \rangle, 
% \]
% and hence the $d$-th elementary ideal $E_d\subset\Z[t^{\pm1}]$ is as follows:
% \begin{align*}
% E_d=\begin{cases}
% (0) & (d=0) \\
% (2t^{-1}-5+2t) & (d=1)\\
% (m,2t-1,2-t) & (d=2)\\
% (1) & (d\geq3).
% \end{cases}
% \end{align*}
% It can not coincide with that of $k(S^1\times S^2)$ for any $k$ and $m$, 
% and hence $X_{\rm (a15)}^0$ is not a shadow of any closed $4$-manifold.
Suppose that there exists a closed $4$-manifold $W$ admitting a shadow $X_{\rm (a15)}^0$. 
Note that $\pi_1(W)\cong \pi_1(X_{\rm (a15)}^0)\cong \langle x,y\mid xyx^{-1}y^{-2}\rangle$, 
which is not a cyclic group. % Baumslag–Solitar Groups
Set $M=\Nbd(X_{\rm (a15)}^0;W)$. 
Its Kirby diagram is shown in the left part of Figure~\ref{fig:Kirby_a15} for some $m\in\Z$. 
Then we have $H_1(\partial M)=\Z$, and hence $\partial M$ must be $S^1\times S^2$. 
Therefore, 
$W$ admits a handle decomposition consisting of one $0$-handle, $2$ $1$-handles, 
one $2$-handles, one $3$-handle and one $4$-handle. 
Considering the dual decomposition, 
we see that $\pi_1(W)$ is generated by one element, 
which is a contradiction. 
% We next consider $X_{\rm (a15)}^1$. 
% A Kirby diagram of its $4$-dimensional thickening is 
% depicted in the right part of Figure~\ref{fig:Kirby_a15}. 
% The second homology vanish, and first homology is isomorphic to $\Z/3\Z$. 
% Therefore, it can not be a shadow of closed $4$-manifold by Lemma~\ref{lem:Costantino}. 
\end{proof}

% - - - - - - - - - - - - - - - - -
\begin{figure}[tbp]
\begin{tabular}{ccc}
\centering
\begin{minipage}[b]{.6\hsize}
\labellist
\footnotesize\hair 2pt
\pinlabel $m$ [Bl] at  85.78 25.51
\pinlabel $m$ [Bl] at  227.51 25.51
\endlabellist
\centering
\includegraphics[width=.7\hsize]{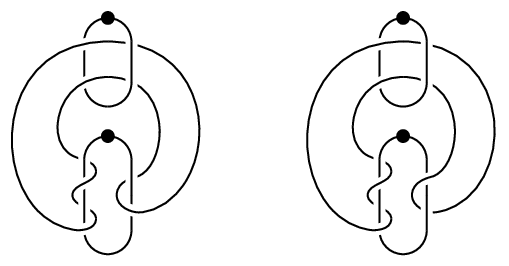}
\caption{The left and right diagrams are Kirby diagrams of $4$-dimensional thickenings of the simple polyhedra $X_{\rm (a15)}^0$ and $X_{\rm (a15)}^1$, respectively.}
\label{fig:Kirby_a15}
\end{minipage}
\begin{minipage}[c]{.03\hsize}
 ${}$
\end{minipage}
\begin{minipage}[b]{.3\hsize}
\labellist
\footnotesize\hair 2pt
\pinlabel $m$ [B] at  62.36 7.92
\endlabellist
\centering
\includegraphics[width=.8\hsize]{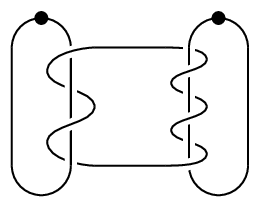}
\caption{A Kirby diagram of a $4$-dimensional thickening of the simple polyhedron $X_{\rm (m2)}$.}
\label{fig:Kirby_m2}
\end{minipage}
\end{tabular}
\end{figure}
% - - - - - - - - - - - - - - - - -

\begin{lemma}
\label{lem:m2}
The simple polyhedron $X_{\rm (m2)}$ is not a shadow of closed $4$-manifolds. 
\end{lemma}
\begin{proof}
Suppose that there exists a closed $4$-manifold $W$ admitting a shadow $X_{\rm (m2)}$. 
Note that $\pi_1(W)\cong \pi_1(X_{\rm (m2)})\cong \langle x,y\mid x^2y^3\rangle$, which is not cyclic. 
Set $M=\Nbd(X_{\rm (m2)};W)$. 
Its Kirby diagram is depicted in Figure~\ref{fig:Kirby_m2}. 
Then 
% we have $H_1(\partial M)=\Z$, and hence $\partial M$ must be $S^1\times S^2$. 
% Therefore, 
% $W$ admits a handle decopmposition consisting of one $0$-handle, $2$ $1$-handles, 
% one $2$-handles, one $3$-handle and one $4$-handle. 
% Considering the dual decomposition, 
% we see that $\pi_1(W)$ is generated by one element, 
% which is a contradiction. 
the lemma can be proved in much the same way as Lemma~\ref{lem:a15}. 
\end{proof}

% - - - - - - - - - - - - - - - - - - - - - - - - - - - - - - - - - - - - - - - - -
\subsection{Classification}
\begin{lemma}
The simple polyhedra $X_{\rm (a1)}$ and $X_{\rm (a2)}$ are shadows only of $S^1\times S^3$. 
\end{lemma}
\begin{proof}
The simple polyhedra $X_{\rm (a1)}$ and $X_{\rm (a2)}$ have 
unique $4$-dimensional thickenings, which are diffeomorphic to $S^1\times B^3$. 
Hence, they are only shadows of $S^1\times S^3$. 
\end{proof}
\begin{lemma}
\label{lem:a3-5}
If a closed 4-manifold $W$ admits a shadow homeomorphic to 
$X_{\rm (a3)}, X_{\rm (a4)}$ or $X_{\rm (a5)}$, 
then $\shco_{1/2}(W)=0$. 
\end{lemma}
\begin{proof}
The simple polyhedra $X_{\rm (a3)}, X_{\rm (a4)}$ and $X_{\rm (a5)}$, respectively, 
collapses onto $S^2$, $\RP^2$ and $S^2$, whose $1/2$-weighted complexities are $0$. 
\end{proof}
\begin{lemma}
The simple polyhedron $X_{\rm (a9)}$ is a shadow only of $\mathcal{S}_3$. 
\end{lemma}
\begin{proof}
We have $\pi_1(X_{\rm (a9)})\cong \Z/3\Z$, $b_2(X_{\rm (a9)})=1$ and $c(X_{\rm (a9)})=0$. 
By Theorem~\ref{thm:finite_pi1}, if $X_{\rm (a9)}$ is a shadow of a closed $4$-manifold, 
it is nothing but $\mathcal{S}_3$. 
Actually, a gleam on $X_{\rm (a9)}$ defined by $\gl(R_1)=1,\gl(R_2)=-1$ and $\gl(R_3)=1$ 
provides $\mathcal{S}_3$, 
where $R_1$ and $R_2$ are two disk regions of $X_{\rm (a9)}$ and 
$R_3$ is a single annular region of $X_{\rm (a9)}$. 
\end{proof}
\begin{lemma}
\label{lem:a10}
If the simple polyhedron $X_{\rm (a10)}$ is a shadow of a closed $4$-manifold $W$, 
then $W$ is $\mathcal{S}_2$ or $\mathcal{S}_2'$. 
\end{lemma}
\begin{proof}
It follows from $\pi_1(X_{\rm (a10)})\cong \Z/2\Z$, $b_2(X_{\rm (a10)})=1$, $c(X_{\rm (a10)})=0$ and Theorem~\ref{thm:finite_pi1}. 
\end{proof}
\begin{remark}
The $1/2$-weighted shadow-complexities of $\mathcal{S}_2$ and $\mathcal{S}_2'$ are actually $1/2$ as shown in Lemma~\ref{lem:m5}. 
We have not proven in the proof of Lemma~\ref{lem:a10} 
that $\mathcal{S}_2$ or $\mathcal{S}_2'$ admits a shadow $X_{\rm (a10)}$, 
so we do not know at this moment if the $1/2$-weighted shadow-complexities of them are exactly $1/2$ or not. 
\end{remark}
\begin{lemma}
If a closed $4$-manifold $W$ admits a shadow $X$ 
homeomorphic to $X_{\rm (a11)}$, $X_{\rm (a12)}$, $X_{\rm (a13)}$ or $X_{\rm (a14)}$, then $\shco_{1/2}(W)=0$. 
\end{lemma}
\begin{proof}
% For $X\in\{X_{\rm (a11)},X_{\rm (a12)},X_{\rm (a13)},X_{\rm (a14)}\}$, 
In each case, we have $\pi_1(W)\cong \pi_1(X)\cong \{1\}$, $b_2(W)\leq b_2(X)\leq2$ and 
$\shco(W)=c(X)=0$. 
Therefore, $\shco_{1/2}(W)=0$ by Theorem~\ref{thm:finite_pi1}. 
\end{proof}
\begin{lemma}
The simple polyhedron $X_{\rm (a16)}$ is shadows only of 
$S^2\times S^2$ and the connected sums of at most $3$ copies in $\{S^4,\CP,\mCP\}$. 
Especially, closed $4$-manifolds with $\shco_{1/2}=1/2$ admitting shadows homeomorphic to $X_{\rm (a16)}$ 
are only $3\CP$, $2\CP\#\mCP$, $\CP\#2\mCP$ and $3\mCP$. 
\end{lemma}
\begin{proof}
Note that $X_{\rm (a16)}$ is homeomorphic to $X_3$ that is the simple polyhedron constructed in Subsection~\ref{subsec:Examples}. 
By Theorem~\ref{thm:finite_pi1}, the lemma follows. 
\end{proof}
% - - - - - - - - - - - - - - - - -
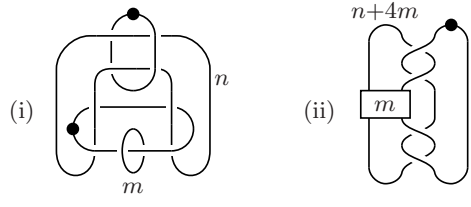
\begin{figure}[tbp]
\labellist
\footnotesize\hair 2pt
\pinlabel $n$ [Bl] at  106.76 62.59
\pinlabel $m$ [B] at    67.24 6.73
% \pinlabel $n$ [Bl] at  390.22 62.59
% \pinlabel $m$ [B] at   350.70 6.73
\pinlabel $m$ [B] at   197.63 49.42
\pinlabel $n{+}4m$ [B] at   197.63 97.19
\pinlabel (i) [Br] at  19.05 45.58
\pinlabel (ii) [Br] at  174.95 45.58
% \pinlabel (iii) [Br] at  302.51 45.58
\endlabellist
\centering
\includegraphics[width=.5\hsize]{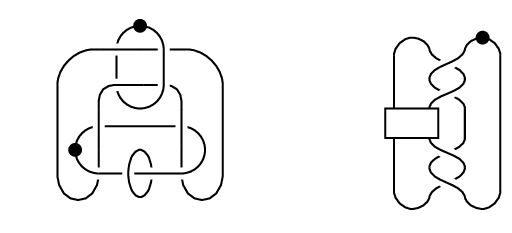}
\caption{
(i) A Kirby diagram of a $4$-dimensional thickening of the simple polyhedron $X_{\rm (a17)}^0$. }
\label{fig:Kirby_a17_0}
\end{figure}
% - - - - - - - - - - - - - - - - -
\begin{lemma}
The simple polyhedron $X_{\rm (a17)}^0$ is shadows only of $S^1\times S^3$, $\CP\#(S^1\times S^3)$ and $\mCP\#(S^1\times S^3)$. 
\end{lemma}
\begin{proof}
% We first consider the polyhedron $X_{\rm (a17)}^0$. 
Let $M_0$ be the $4$-dimensional thickening of $X_{\rm (a17)}^0$ equipped with arbitrary gleam. 
A Kirby diagram of $M_0$ is shown in Figure~\ref{fig:Kirby_a17_0}-(i), where $m,n$ are some integers. 
The attaching circle with framing $m$ is canceled with a dotted circle, 
so that we get a Kirby diagram shown in Figure~\ref{fig:Kirby_a17_0}-(ii). 
By replacing the dotted circle in the figure with a $0$-framed knot, 
we get a surgery diagram of the boundary $\partial M_0$. 
By Wu's result \cite[Theorem 5.1]{Wu99}, 
$\partial M_0$ is not homeomorphic to $k(S^1\times S^2)$ for any $k$ unless $m=0$. 
Suppose $m=0$. 
The $4$-manifold $M_0$ admits a Kirby diagram given by a $2$-component unlink consisting of one dotted circle and one unknot with framing coefficient $n$. 
Therefore, $X_{\rm (a17)}^0$ can be embedded in $S^1\times S^3$, $\CP\#(S^1\times S^3)$ and $\mCP\#(S^1\times S^3)$ as shadows. 
\end{proof}
\begin{lemma}
The simple polyhedron $X_{\rm (a17)}^1$ is shadows only of $S^1\times S^3$, $\CP\#(S^1\times S^3)$ and $\mCP\#(S^1\times S^3)$. 
\end{lemma}
\begin{proof}
 % - - - - - - - - - - - - - - - - -
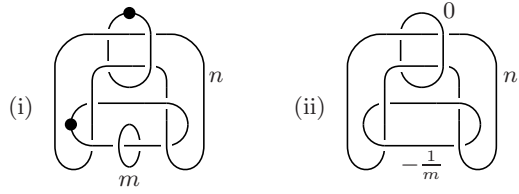
\begin{figure}[tbp]
\labellist
\footnotesize\hair 2pt
\pinlabel $n$ [Bl] at   106.76 62.59
\pinlabel $m$ [B] at     67.24  6.73
\pinlabel (i) [Br] at    19.05 45.58
\pinlabel (ii) [Br] at  174.95 45.58
\pinlabel $n$ [Bl] at   263.66 62.59
\pinlabel $-\frac1m$ [B] at 223.14 15.23
\pinlabel $0$ [Bl] at   231.65 96.60
\endlabellist
\centering
\includegraphics[width=.55\hsize]{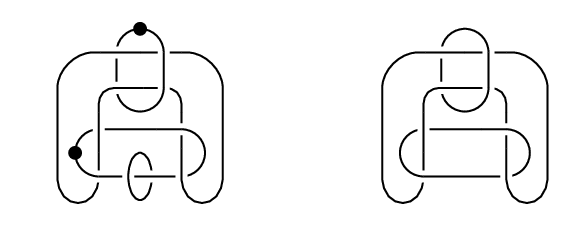}
\caption{
(i) A Kirby diagram of a $4$-dimensional thickening of the simple polyhedron $X_{\rm (a17)}^1$. 
(ii) A surgery diagram of the boundary of the $4$-dimensional thickening of the simple polyhedron $X_{\rm (a17)}^1$. }
\label{fig:Kirby_a17_1}
\end{figure}
% - - - - - - - - - - - - - - - - -
% We next consider the polyhedron $X_{\rm (a17)}^1$. 
Let $M_1$ be the $4$-dimensional thickening of $X_{\rm (a17)}^1$ equipped with arbitrary gleam, 
which is represented by a Kirby diagram shown in Figure~\ref{fig:Kirby_a17_1}-(i) for some $m,n\in\Z$. 
By replacing the dotted circles with $0$-framed unknots, 
we obtain a surgery diagram of the $3$-manifold $\partial M_1$. 
Performing a slum-dunk move once, we obtain the diagram of $\partial M_1$ shown in Figure~\ref{fig:Kirby_a17_1}-(ii). 
By an explicit computation from this diagram, 
% a presentation of the fundamental group of $\partial M_1$ can be obtained as 
we have 
\[
\pi_1(\partial M_1)\cong
\langle
x,y,z
\mid
[x,z],[z,y^{-1}xy],x^nzyzy^{-1},z^{-1}(xy^{-1}xy)^m
\rangle. 
\]

% We now suppose $\partial M_1\cong k(S^1\times S^2)$ for any $k$, 
% We then derive necessary conditions of $n$ and $m$ for $\partial M_1$ to be homeomorphic to $k(S^1\times S^2)$ for some $k\in\Z_{\geq0}$. 
Note that 
\[
H_1(\partial M_1)\cong
\left\{
\begin{array}{ll}
\Z\langle y\rangle & (4m+n=\pm1)\\
\Z\langle y\rangle\oplus\Z\langle x\rangle & (4m+n=0)\\
\Z\langle y\rangle\oplus(\Z/(4m+n)\Z\langle x\rangle) & (\text{otherwise}). 
\end{array}
\right.
\]
Therefore, in order for $\partial M_1$ to be homeomorphic to $k(S^1\times S^2)$ for some $k\in\Z_{\geq0}$, 
it is necessary that $4m+n=\pm1$ or $0$. 

Suppose $4m+n=\pm1$. 
% ,  $\partial M_1\cong S^1\times S^2$. 
By explicit calculations from the presentation of $\pi_1(\partial M_1)$, we have 
\[
E_d(\pi_1(\partial M_1))\cong
\left\{
\begin{array}{ll}
(0) & (d=0)\\
(n+m(1+t)(1+t^{-1})) & (d=1)\\
(1) & (d\geq2). 
\end{array}
\right.
\]
By Lemma~\ref{lem:elementary_ideal}, we need $m=0$ and $n=\pm1$. 
Conversely, substituting $m=0$ and $n=\pm1$ into the diagram shown in Figure~\ref{fig:Kirby_a17_1}-(i), 
we obtain a Kirby diagram given by a $2$-component unlink consisting of one dotted circle and one unknot with framing $\pm1$ after easy Kirby calculus. 
It implies that $X_{\rm (a17)}^1$ can be embedded in $\CP\#(S^1\times S^3)$ and $\mCP\#(S^1\times S^3)$ as shadows. 

Suppose $4m+n=0$. 
% , which implies $\partial M_1\cong 2(S^1\times S^2)$. 
By explicit calculations from the presentation of $\pi_1(\partial M_1)$, the Alexander matrix is given as 
\[
\left(
\begin{array}{ccc}
1-t_2^{2m} & 0 & t_2(1-t_2) \\
t_1^{-1}t_2^{2m}(1-s^{2m}) & t_1^{-1}(1-t_2)(1-t_2^{2m}) & 1-t^2 \\
\frac{1-t_2^n}{1-t_2} & t_2^{2m+n}(1-t_2^{2m}) & t_2^{n}(1+t_1t_2^{2m})\\
t_2^{-2m}(1+t_1^{-1}t_2)\frac{1-t_2^{2m}}{1-t_2} & t_1^{-1}t_2^{1-2m}(t_2-1)\frac{1-t_2^{2m}}{1-t_2} & -t_2^{-2m}
\end{array}
\right), 
\]
where 
$t_1$ and $t_2$, respectively, are the images of $y$ and $x$ by the homomorphism $\Z\pi_1(\partial M_1)\to\Z[t_1^{\pm1},t_2^{\pm1}]$ 
induced by the abelianization $\pi_1(\partial M_1)\to H_1(\partial M_1)$. 
The upper-right $2\times2$-minor is $t_1^{-1}t_2(1-t_2)^2(1-t_2^{2m})$, 
and hence $((1-t_2)^2(1-t_2^{2m}))\subset E_1(\pi_1(\partial M_1))$. 
By Lemma~\ref{lem:elementary_ideal}, $E_1(\pi_1(\partial M_1))$ must be $(0)$, 
so we need $m=0$. 
Since $4m+n=0$, we have $n=0$. 
Conversely, substituting $m=n=0$ into the diagram shown in Figure~\ref{fig:Kirby_a17_1}-(i), 
we obtain a diagram given by $2$-component unlink consisting of one dotted circle and one unknot with framing $0$ after easy Kirby calculus. 
It implies that $X_{\rm (a17)}^1$ can be embedded in $S^1\times S^3$ as a shadow. 
\end{proof}
\begin{lemma}
The simple polyhedron $X_{\rm (m1)}$ is a shadow only of $S^1\times S^3$. 
\end{lemma}
\begin{proof}
The simple polyhedron $X_{\rm (m1)}$ have a 
unique $4$-dimensional thickening, which is $S^1\times B^3$. 
Hence, it is a shadow only of $S^1\times S^3$. 
\end{proof}
\begin{lemma}
\label{lem:m5}
The simple polyhedron $X_{\rm (m5)}$ is shadows only of $\mathcal{S}_2$ and $\mathcal{S}_2'$. 
\end{lemma}
\begin{proof}
We have $\pi_1(X_{\rm (m5)})\cong \Z/2\Z$ and $b_2(X_{\rm (a9)})=1$. 
By Theorem~\ref{thm:finite_pi1}, if $X_{\rm (m5)}$ is a shadow of a closed $4$-manifold, 
it is nothing but $\mathcal{S}_2$ or $\mathcal{S}_2'$. 
Actually, a gleam on $X_{\rm (a9)}$ defined by $\gl(R_1)=1,\gl(R_2)=-1$ and $\gl(R_3)=1$ 
gives $\mathcal{S}_2$, 
where $R_1$ and $R_2$ are two disk regions of $X_{\rm (a9)}$ and 
$R_3$ is the annular region of $X_{\rm (a9)}$. 
If we equip $X_{\rm (m5)}$ with gleams $\gl(R_1)=1,\gl(R_2)=-1$ and $\gl(R_3)=0$, 
it yields $\mathcal{S}_2'$. 
\end{proof}

\begin{theorem}
\label{thm:complexity1/2}
The $1/2$-weighted shadow-complexity of a closed $4$-manifold $W$ is $1/2$ 
if and only if 
$W$ is diffeomorphic to either one of $3\CP$, $2\CP\#\mCP$, $\CP\#2\mCP$, $3\mCP$, 
$S^1\times S^3$, $(S^1\times S^3)\#\CP$, $(S^1\times S^3)\#\mCP$, $\mathcal{S}_2$, $\mathcal{S}_2'$ or $\mathcal{S}_3$. 
\end{theorem}
%==================================================================================


\begin{thebibliography}{99}
\bibitem{Cos05}
F. Costantino, 
\textit{Shadows and branched shadows of 3 and 4-manifolds}.
Scuola Normale Superiore, Edizioni della Normale, Pisa, Italy, 2005. 

\bibitem{Cos06}
F. Costantino,
\textit{Stein domains and branched shadows of $4$-manifolds},
Geom. Dedicata {\bf 121} (2006), 89--111.

\bibitem{Cos06b}
F. Costantino,
\textit{Complexity of $4$-manifolds}, 
Exp. Math. {\bf 15} (2006), no. 2, 237--249.

\bibitem{Cos08}
F. Costantino,
\textit{Branched shadows and complex structures on $4$-manifolds},
J. Knot Theory Ramif. {\bf 17} (2008), no. 11, 1429--1454.

\bibitem{CT08}
F. Costantino and D.~Thurston,
\textit{$3$-manifolds efficiently bound $4$-manifolds},
J. Topol. {\bf 1} (2008), no. 3, 703--745. 

\bibitem{GK16}
D. Gay and R. Kirby, 
\textit{Trisecting $4$-manifolds}, 
Geom. Topol. {\bf 20} (2016), no. 6, 3097--3132.

\bibitem{GS99}
R. Gompf and A. Stipsicz, 
\textit{$4$-manifolds and Kirby calculus}, 
Grad. Studies in Math., 20, Amer. Math. Soc., Providence, 1999.

\bibitem{IK17}
M.~Ishikawa and Y.~Koda, 
\textit{Stable maps and branched shadows of $3$-manifolds},
Math. Ann. {\bf 367} (2017), no. 3-4, 1819--1863.

\bibitem{KMN18}
Y. Koda, B. Martelli and H. Naoe, 
\textit{Four-manifolds with shadow-complexity one}, 
Ann. Fac. Sci. Toulouse Math. (6), \textbf{31} (2022), no. 4, 1111--1212. 

\bibitem{KN20}
Y. Koda and H. Naoe, 
\textit{Shadows of acyclic $4$-manifolds with sphere boundary},
Algebr. Geom. Topol. {\bf 20} (2020), no. 7, 3707--3731

\bibitem{LP72}
F. Laudenbach, V. Po\'{e}naru, 
\textit{A note on $4$-dimensional handlebodies}, 
Bull. Soc. Math. France {\bf 100} (1972), 337--344.

\bibitem{Mar05} 
B. Martelli, 
\textit{Links, two-handles, and four-manifolds}, 
Int. Math. Res. Not. IMRN  {\bf 2005},  no. 58, 3595--3623. 

\bibitem{Mar11} 
B. Martelli, 
\textit{Four-manifolds with shadow-complexity zero}, 
Int. Math. Res. Not. IMRN  {\bf 2011},  no. 6, 1268--1351. 

\bibitem{Mei18}
J. Meier, 
\textit{Trisections and spun four-manifolds},
Math. Res. Lett. {\bf 25} (2018), no. 5, 1497--1524. 

\bibitem{MSZ16}
J. Meier, T. Schirmer and A. Zupan, 
\textit{Classification of trisections and the generalized property R conjecture}, 
Proc. Amer. Math. Soc. {\bf 144} (2016), no.11, 4983–--997.

\bibitem{MZ17}
J. Meier and A. Zupan, 
\textit{Genus-two trisections are standard}, 
Geom. Topol. {\bf 21} (2017), no. 3, 1583--1630. 

\bibitem{MZ18}
J. Meier and A. Zupan, 
\textit{Bridge trisections of knotted surfaces in 4-manifolds},
Proc. Natl. Acad. Sci. USA {\bf 115} (2018), no. 43, 10880--10886.

\bibitem{Nao17}
H. Naoe, 
\textit{Shadows of $4$-manifolds with complexity zero and polyhedral collapsing}, 
Proc. Amer. Math. Soc. {\bf 145} (2017), no. 10, 4561--4572. 

\bibitem{Nao23} 
H. Naoe, 
\textit{The special shadow-complexity of $\#k(S^1\times S^3)$}, 
preprint. arXiv:2309.09225.

% \bibitem{Pao77}
% P. S. Pao, 
% \textit{The topological structure of $4$-manifolds with effective torus actions. I}, 
% Trans. Amer. Math. Soc. {\bf 227} (1977), 279--317.

\bibitem{Tur94}
V.G. Turaev, 
\textit{Quantum invariants of knots and $3$-manifolds}, 
De Gruyter Studies in Mathematics, vol 18, Walter de Gruyter \& Co., Berlin, 1994.

\bibitem{Wil20}
M. Williams, 
\textit{Trisections of Flat Surface Bundles over Surfaces}, 
PhD Thesis, The University of Nebraska (2020), 81 pp.

\bibitem{Wu99}
Y.-Q. Wu, 
\textit{Dehn Surgery on Arborescent Links}, 
Trans. Am. Math. Soc. \textbf{351} (1999),
no. 6, p. 2275--2294.

\end{thebibliography}
\end{document}